\documentclass{article}
\usepackage[a4paper]{geometry} 
\usepackage[utf8]{inputenc} 
\usepackage[T1]{fontenc}

\usepackage{amsthm}
\usepackage{amsmath}
\usepackage{amssymb}
\usepackage{mathtools} 
\usepackage{dsfont}
\usepackage{algorithm}
\usepackage{algorithmic}
\usepackage{siunitx}

\usepackage{indentfirst}
\usepackage{comment}
\usepackage{mleftright} 
\usepackage{bropd} 
\usepackage{xspace}
\usepackage[english]{babel} 
\usepackage{hyperref}

\usepackage{graphicx} 
\usepackage{subcaption}

\usepackage[normalem]{ulem} 
\usepackage[
    backend=bibtex,
    style=authoryear,
    natbib=true,
    date=year,
    url=false,
    doi=false,
    eprint=true,
    isbn=false,
    useprefix=true,
    maxbibnames=10,
    maxcitenames=2
]{biblatex}
\usepackage{authblk} 
\newtheorem{lemma}{Lemma}
\newtheorem{thm}{Theorem}
\newtheorem{prop}{Proposition}

\newtheorem{rk}{Remark}

\NewDocumentCommand{\Esp}{}{\mathbb{E}}
\NewDocumentCommand{\Prob}{}{\mathbb{P}}
\NewDocumentCommand{\Var}{}{\mathrm{Var}}
\NewDocumentCommand{\Cov}{}{\mathrm{Cov}}
\NewDocumentCommand{\A}{}{\mathcal{A}}
\NewDocumentCommand{\B}{}{\mathcal{B}}
\NewDocumentCommand{\C}{}{\mathcal{C}}
\NewDocumentCommand{\I}{}{\mathcal{I}}
\NewDocumentCommand{\J}{}{\mathcal{J}}
\NewDocumentCommand{\M}{}{\mathcal{M}}
\NewDocumentCommand{\N}{}{\mathcal{N}}
\NewDocumentCommand{\R}{}{\mathcal{R}}
\NewDocumentCommand{\Rbb}{}{\mathbb{R}}
\NewDocumentCommand{\tauH}{}{\tau_{H}}
\NewDocumentCommand{\indep}{}{\perp \!\!\! \perp}

\NewDocumentCommand{\Fn}{o}{\IfNoValueTF{#1}{\hat{F}_n}{\hat{F}_{n_#1}}}
\NewDocumentCommand{\Sn}{o}{\IfNoValueTF{#1}{\hat{S}_n}{\hat{S}_{n_#1}}}
\NewDocumentCommand{\Gn}{o}{\IfNoValueTF{#1}{\hat{G}_n}{\hat{G}_{n_#1}}}

\NewDocumentCommand{\w}{}{\omega}
\NewDocumentCommand{\wi}{}{\omega_i}

\NewDocumentCommand{\whi}{}{\hat{\omega}_i}
\NewDocumentCommand{\whj}{}{\hat{\omega}_j}

\NewDocumentCommand{\Dn}{o}{\IfNoValueTF{#1}{D_n}{D_{n_#1}}}
\NewDocumentCommand{\Dm}{}{D_m}
\NewDocumentCommand{\indicator}{m}{\mathds{1}\{#1\}}
\NewDocumentCommand{\muT}{}{\mu_{\tau}}
\NewDocumentCommand{\muTtilde}{}{\tilde{\mu}_{\tau}}
\NewDocumentCommand{\muThat}{o}{\IfNoValueTF{#1}{\hat{\mu}_{\tau,n}}{\hat{\mu}_{\tau,n_#1}}}

\NewDocumentCommand{\Rnii}{}{\hat{\R}_{n_2}^*}
\NewDocumentCommand{\RniiG}{}{\hat{\R}_{n_2}^{\hat{G}}}
\NewDocumentCommand{\qhat}{}{\hat{q}_{n_2}}
\NewDocumentCommand{\Csplit}{}{\C^{\text{split}}_n}
\NewDocumentCommand{\Croo}{}{\C^{\text{roo}}_n}
\NewDocumentCommand{\intZ}{}{\int_{\Rbb^d}}
\NewDocumentCommand{\intT}{}{\int_0^{\tau}}

\title{A Comprehensive Framework for Evaluating Time to Event Predictions using the Restricted Mean Survival Time}

\author[*]{Ariane Cwiling}
\author[*]{Vittorio Perduca}
\author[*]{Olivier Bouaziz}
\affil[*]{Université Paris Cité, CNRS, MAP5, F-75006 Paris, France}
\date{}

\begin{document}

\maketitle

{\bf Abstract.} The restricted mean survival time (RMST) is a widely used quantity in survival analysis due to its straightforward interpretation. For instance, predicting the time to event based on patient attributes is of great interest when analyzing medical data. In this paper, we propose a novel framework for evaluating RMST estimations. A criterion that estimates the mean squared error of an RMST estimator using Inverse Probability Censoring Weighting (IPCW) is presented. A model-agnostic conformal algorithm adapted to right-censored data is also introduced to compute prediction intervals and to evaluate local variable importance. Finally, a model-agnostic statistical test is developed to assess global variable importance. Our framework is valid for any RMST estimator that is asymptotically convergent and works under model misspecification.
\smallskip


\bigskip


\section{Introduction}

In survival analysis, in the context of right-censored data, it is common to model the effect of covariates on the hazard rate, using the ubiquitous Cox model~\citep[see][]{cox_partial_1975}. This model is interpreted in terms of hazard ratios and is widely used to analyze incomplete data in medical applications. However, it relies on the strong assumption of proportional hazard (PH). As a result, approaches that focus on other estimands, such as the restricted mean survival time (RMST), have been proposed. This quantity represents the expected duration of the minimum between the occurrence of an event of interest and a predefined time horizon. It is clinically meaningful (as an expected time) and has gained considerable attention in recent years due to its simple interpretation. While initial works on this topic still relied on the PH assumption~\citep[see][]{karrison_restricted_1987,zucker_restricted_1998}, new approaches have been developed to directly model the RMST without making any assumptions~\citep[see][]{andersen_regression_2004, tian_predicting_2014, wang_modeling_2018}.

Also, in time to event analysis, it might be of interest to predict the future occurrence of the event of interest using an estimation of the RMST. This is the case for instance, when clinicians aim at predicting the time to relapse, cancer occurrence or death of a patient. In recent years, new methods have been developed in this context~\citep[see][]{zhao_deep_2021} and there is thus a need for evaluating the performance of those learning models. This is usually a challenge in the presence of right-censored data because the censored times are not observed and it is therefore difficult to assess the performance of the learning model for those times. To address this challenge, the C-index~\citep[see][]{heagerty_survival_2005} has emerged as a widely used metric, particularly as it has been adapted to censored data~\citep[see][]{gerds_estimating_2013}. However, it has been shown not to be a proper scoring function~\citep[see][]{blanche_c-index_2019}. In contrast, the time-dependent area under the ROC curve is a proper score, but it is also based on the evaluation of the rank preservation of the predictions. When it comes to quantitative measures, the mean squared error (MSE) is a proper one, but it is not readily available due to censoring. 

Another important topic in this context concerns the construction of prediction intervals, which evaluate the degree of confidence in the prediction by taking into account the individual variability. The conformal approach originally proposed by \citet{vovk_algorithmic_2005} and later expanded and popularized by \citet{lei_distribution-free_2018}, offers a way to build prediction intervals with guaranteed coverage. Although this approach has been adapted to right-censored data, it is still subject to significant constraints.
\citet{bostrom_predicting_2019}, \citet{chen_deep_2020}, \citet{teng_t-sci_2021} proposed model-specific conformal inference algorithms for the Random Survival Forest (RSF), DeepHit, and Cox-MLP, respectively. \citet{candes_conformalized_2023} proposed a model-agnostic algorithm to build a prediction lower bound for right-censored data, but with no upper bound and only for censoring of type I.

Finally, being able to interpret the output from a learning model is crucial, especially when using black-box models. To that end, one possibility is to determine the variables' importance by using permutations as developed in \citet{breiman_random_2001} in the context of random forests. This technique is widely used and has been extended to right-censored data, for instance in~\cite{ishwaran_random_2008} and~\cite{zhao_deep_2021}. Recently, model-agnostic importance measures such as LIME and SHAP have been adapted to the estimation of the survival function \citep[see][]{kovalev_survlime_2020,krzyzinski_survshapt_2023} but no extensions for the RMST have been developed yet. Although the leave-one-covariate-out (LOCO) conformal approach introduced by \citet{lei_distribution-free_2018} provides an alternative method for exploring variable importance, it has not yet been extended to right-censored data.

In this work, our goal is to propose a new framework for evaluating time to event predictions, relying on Inverse Probability Censoring Weighting (IPCW) to take into account right-censoring. First, we derived a measure designed to approximate the MSE of an RMST estimator. This measure is similar to the Brier score introduced in \citet{gerds_consistent_2006} to approximate the MSE of a survival function estimator. This idea was introduced in \citet{wang_modeling_2018} for RMST estimation based on generalized linear models. In this work, we extend the measure to any RMST estimator and derive general consistency guarantees.
In addition, a new conformal algorithm for the construction of prediction intervals for restricted times to event is developed, which is inspired by the split algorithm in \citet{lei_distribution-free_2018}. It is further extended to study local and global variable importance within a learning model. In particular, a statistical test for global variable importance is proposed. Those methods are based on the LOCO procedure from \citet{lei_distribution-free_2018}.
All our methods are proved to be asymptotically valid. They are illustrated in simulations and in a real data analysis on breast cancer.
An R package encapsulating our new methods is available at \href{https://github.com/ariane-cwi/conformal}{https://github.com/ariane-cwi/conformal}.

In Section~\ref{sec::notations} we give the main notations used in the following sections. In Sections~\ref{sec::mse},~\ref{sec::predintervals},~\ref{sec::varimportance} we present the new methods described previously, respectively the mean squared error measure, the prediction intervals algorithms and the variable importance measures. The results on simulations are presented in Section~\ref{sec::simulations} and on real data in Section~\ref{sec::realdata}.


\section{Notations and assumptions} \label{sec::notations}

In the context of right-censored data, we denote by $T^*$ the variable of interest, $C$ the censoring time, $T=\min(T^*,C)$ the observed variable and $\delta=\indicator{T^*\leq C}$ the censoring indicator. An observation is then represented by the vector $O = (T,\delta,Z)$ where $Z\in\Rbb^d$ is a covariate vector. We introduce the following notations for the cumulative distribution and survival functions: $F$, $G$, $H$ and $L$ are the cumulative distribution functions of $T^*$, $C$, $T$ and $Z$, respectively. $S=1-F$ is the survival function of $T^*$. For all these functions, the same notations are used for the joint and conditional cumulative distribution functions with respect to $Z$, for instance $F(t \mid z) = \Prob(T^* \leq t \mid Z = z) = 1 - S(t \mid z)$. Finally, we note $P(t,\cdot,z) = \sum_{\delta=0,1}\Prob(T \leq t,\delta, Z \leq z)$. 

Let $\tauH = \inf\{t>0 : 1-H(t \mid Z)= 0 \text{ a.s.} \}$. The RMST is defined for a fixed time horizon $\tau \leq\tauH$ as
\begin{equation}\label{eq::RMSTdef}
     \muT^*(Z) = \Esp[T^* \wedge \tau \mid Z] = \intT S(t \mid Z) dt.
\end{equation}
We suppose that the i.i.d. observations are divided into a training set $\Dn = \{O_i: i \in \I\}$ of size $n$ and a test set $\Dm = \{O_j: j \in \J\}$ of size $m$. We will call a learning model/algorithm $\A$, a function $\A: \Dn \mapsto \muThat \in \M_\A$ that maps a training set $\Dn$ to an RMST estimator $\muThat$, also referred to as a predictor. $\M_\A$ is the space of possible outcomes $\muT(\cdot):\Rbb^d \to\Rbb_+$ of $\A$. We impose that these functions verify $\sup_{z \in \Rbb^d} | \muT(z) | \leq K$ for a positive constant $K$. In the following, when the split conformal intervals will be introduced (see Section~\ref{sec::predintervals}), the training set will be further split into two parts, that is, $\I$ will be divided into two subsets $\I_1$ and $\I_2$ of sizes $n_1$ and $n_2$ respectively, such that $n_1 + n_2 = n$. In that case, $\Dn[k] = \{O_i: i \in \I_k\}$, $k=1,2$, will be the corresponding subsets of $\Dn$. Estimators of the functions $F$, $S$, $G$ and $\muT^*$ computed on the training set $\Dn$ are written $\Fn$, $\Sn$, $\Gn$, and $\muThat$ respectively. If they are computed on one of the subsets $\Dn[k]$, $k=1,2$, they are denoted $\Fn[k]$, $\Sn[k]$, $\Gn[k]$ and $\muThat[k]$, respectively.

Unless mentioned otherwise, we will assume conditional independence in the following sense:
\begin{equation}\label{eq::conditionalindep}
    T^* \indep C \mid Z.
\end{equation}
We will also assume the RMST estimator to be convergent in the following sense: for all $\tau \leq \tauH $, there exists $\muTtilde \in \M_\A$ such that 
\begin{equation} \label{eq::convergence}
    \intZ | \muThat(z) - \muTtilde(z) | L(dz) \xrightarrow[n \to \infty]{} 0 \; \text{in probability.}
\end{equation}

Moreover we say that the model $\A$ is correctly specified if $\muT^* \in \M_\A$ and $\muTtilde = \muT^*$. It should be noted that this assumption allows for model misspecification since we do not impose $ \muTtilde$ to be equal to the true RMST $\muT^*$ (and in practice this will usually not be the case). 
Typical RMST estimators are mentioned in the introduction. Moreover, in Sections~\ref{sec::simulations} and~\ref{sec::realdata}, estimators obtained after integrating the survival function are implemented: they are based on different estimators of the survival function such as the Kaplan-Meier estimator, the Cox model or the RSF. A direct estimator of the RMST is also constructed from a linear model applied to the pseudo-observations~\citep[see][]{andersen_regression_2004}.

Finally, we will assume the censoring estimator $\Gn$ to be consistent using two different definitions. A censoring estimator $\Gn$ is said to be \textit{strongly consistent} if for all $\tau < \tauH$,
\begin{equation} \label{eq::strongconsistency}
    \sup_{s \leq \tau, z \in \Rbb^d} \big| \Gn(s \mid z) - G(s \mid z) \big| \xrightarrow[n \to \infty]{} 0 \text{ a.s. }
\end{equation}
and \textit{weakly consistent} if for all $\tau < \tauH$,
\begin{equation} \label{eq::weakconsistency}
    \intZ \int_0^{\tau} \big| \Gn(s \mid z) - G(s \mid z) \big| P(ds,\cdot,dz) \xrightarrow[n \to \infty]{} 0 \text{ a.s.} 
\end{equation}
We emphasize that, in those two definitions, we impose the censoring estimator to converge towards the true $G$. This is a strong assumption since it imposes that the censoring model is correctly specified. The censoring function $G$ can be estimated by considering $C$ as the variable of interest. For instance, if the censoring $C$ is independent from the time to event $T^*$ and from the covariates $Z$, then the Kaplan-Meier is a strongly consistent estimator. If the censoring depends on the covariates, it becomes necessary to model the conditional distribution of $C$ accordingly. Various modeling options are explored in~\citet{gerds_consistent_2006}, including the Cox model, the Aalen additive model, and the kernel-type model introduced in~\citet{dabrowska_uniform_1989}. Alternatively, approaches like the single-index method proposed in~\citet{bouaziz_conditional_2010} or the RSF method outlined in~\citet{ishwaran_random_2008} can also be applied.

\section{Performance criterion for the estimation of the RMST} \label{sec::mse}

When the data are fully observed, a classical quantity to measure the prediction performance of an estimator $\muThat$ is the Residual Sum of Squares (RSS)
\begin{equation}\label{eq::RSS}
    \text{RSS}(\muThat,\Dm) = \frac{1}{m} \sum_{j \in \J} \Big(T_j^* \wedge \tau - \muThat(Z_j) \Big)^2,
\end{equation}
where $\muThat$ is computed on the training set and the RSS is evaluated on the test set. When the RMST estimator is convergent (see Equation~\eqref{eq::convergence}) then, as $n$ and $m$ go to infinity, the RSS will converge to the Mean Squared Error (MSE) 
\[
    \text{MSE}(\muTtilde) = \Esp\Big[\big(T^* \wedge \tau-\muTtilde(Z)\big)^2 \Big].
\]
However, in the context of right-censored data, the event times are not all observed and the score~\eqref{eq::RSS} cannot be computed. This issue has been addressed in~\cite{gerds_consistent_2006} when the goal is to estimate the survival function. In our work, we extended this approach to estimate the MSE of an RMST estimator, based on IPCW, similarly to the MSD criterion of~\citet{wang_modeling_2018}. We define:
\begin{equation} \label{eq::WRSSdef}
    \text{WRSS}(\muThat, \Gn, \Dm) = \frac{1}{m} \sum_{j \in \J} \Big(T_j \wedge \tau - \muThat(Z_j) \Big)^2 \whj,
\end{equation}
where
\begin{equation} \label{eq::weights}
    \whj = \frac{\indicator{T_j \leq \tau} \delta_j}{1-\Gn(T_j- \mid Z_j)} + \frac{\indicator{T_j > \tau}}{1-\Gn(\tau \mid Z_j)},
\end{equation}
and $\Gn$ is a consistent estimator of the censoring cumulative distribution function. We have the following result.

\begin{thm} \label{thm::WRSS}
Let $\Gn$ be a consistent estimator in the weak sense, as defined by Equation~\eqref{eq::weakconsistency}.
Then, under conditional independence (see Equation~\eqref{eq::conditionalindep}) we have:
\[
    \sup_{\tau \leq \tauH } |\text{WRSS}(\muThat, \Gn, \Dm) - \text{MSE}(\muTtilde) | \to 0 \text{ as } n,m \to \infty \; \text{in probability.}
\]
\end{thm}

Note that the validity of the result relies on Equation~\eqref{eq::weakconsistency}. If the latter is verified with the estimator of the censoring distribution function computed on the training set only, then it also holds on the pooled data including both the training and the test set. In applications we recommend using all the data to estimate $G$. Also, a straightforward result from Theorem~\ref{thm::WRSS} is that our WRSS estimator is asymptotically equivalent to the RSS estimator defined in Equation~\eqref{eq::RSS} since both estimators converge towards $\text{MSE}(\muTtilde)$. 

Finally, we recall that the mean squared error can be decomposed in the following way:
\begin{equation}\label{eq::MSEdecomposition}
    \text{MSE}(\muTtilde) = \Esp\Big[\big(\muT^*(Z) - \muTtilde(Z)\big)^2\Big] + \Esp \Big[\big( T^* \wedge \tau - \muT^*(Z)\big)^2 \Big],
\end{equation}
where on the right-hand side of the equation, the first quantity represents an imprecision term and the second one, an inseparability term~\citep[see][]{gerds_consistent_2006}.
If the model is correctly specified, i.e. $\muTtilde = \muT^*$, then the imprecision term will vanish. In that case, the WRSS estimator will converge to the inseparability term, as $n$ and $m$ go to infinity.


\section{Prediction intervals} \label{sec::predintervals}

In this section, we explain how a prediction interval can be built from the RMST estimator. Our method is based on the conformal intervals method, developed initially by~\citet{vovk_algorithmic_2005} and extended, among others, by~\citet{lei_distribution-free_2018}. In the latter article, the authors provide algorithms to construct prediction intervals that have finite-sample validity without making any assumptions about the distribution of the observations. 
More specifically, for a given confidence level $1-\alpha$ and a new individual with covariate $Z$, the aim is to construct a prediction interval $\C(Z)\subseteq \Rbb$ such that 
\[
     \Prob\big(T^* \wedge \tau \in \C(Z)\big) \geq 1 - \alpha.
\]
Our method adapts the conformal prediction approach from~\citet{lei_distribution-free_2018} using the IPCW technique to deal with right-censored data.

\subsection{IPCW Split Conformal algorithm} \label{sec::split}

Since the original conformal prediction algorithm is computationally intensive, we rely on the so-called \textit{split conformal prediction} developed by \citet{lei_distribution-free_2018} as an alternative approach. Its computational cost is a small fraction of the full conformal method and finite-sample guarantees are very similar. It operates as follows. 
First divide $\I$ into two subsets $\I_1$ and $\I_2$ of size $n_1$ and $n_2$ respectively, such that $n_1 + n_2 = n$. Let $\Dn[k] = \{O_i: i \in \I_k\}$, $k=1,2$, be the corresponding subsets of $\Dn$. Train the learning algorithm on $\Dn[1]$. The resulting estimator $\muThat[1]$ provides predictions for the data in $\Dn[2]$, namely $\{\muThat[1](Z_i)$, $i \in \I_2\}$. 

For ease of presentation, consider first a situation where the data are uncensored. In this setting, the residuals 
\[
    R_i^* = |T_i^* \wedge \tau - \muThat[1](Z_i)|, i \in \I_2,
\] 
can be directly computed from the data. As a result, the cumulative distribution function of the residuals, defined for all $t \geq 0$ by $\R^*(t) := \Prob(R^* \leq t)$, can be approximated by the empirical estimator
\[
\Rnii(t) = \frac{1}{n_2} \sum_{i \in \I_2} \indicator{R_i^* \leq t}.
\]
Finally, the prediction interval for a new individual with covariate $Z$ is defined as
\[
    \Csplit(Z) = [\muThat[1](Z) - \qhat^*,\muThat[1](Z) + \qhat^*],
\]
with $\qhat^* = \inf\{t:\Rnii(t) \geq 1-\alpha \} =  R^*_{(\lceil n_2(1-\alpha) \rceil)}$ the $1 - \alpha$ quantile of the empirical distribution $\Rnii$, where $R^*_{(1)}\leq\ldots\leq R^*_{(n_2)}$ denotes the order statistics of $R^*_{1},\ldots R^*_{n_2}$. By exchangeability between the residual at $(T^*,Z)$ and those at $(T^*_i,Z_i), i \in \I_2$, we have
\[
\Prob(T^* \wedge \tau \in \Csplit(Z))  = \Prob(R^* \leq \qhat^*) \geq 1 - \alpha. 
\]

However, when the data are censored, the residuals $R_i^*$ can no longer be computed. To that purpose, we introduce the residuals of the observed times
\[
    R_i = |T_i \wedge \tau - \muThat[1](Z_i)|, i \in \I_2,
\]
and we adjust the estimator of the cumulative distribution function of these residuals by IPCW, using the same weights as in Equation~\eqref{eq::weights}. The distribution estimator now becomes
\[
\RniiG(t) = \frac{1}{\sum_{i \in \I_2} \whi} \sum_{i \in \I_2} \indicator{R_i \leq t} \whi.
\]

We present below the algorithm that summarizes the procedure. 
Then, Theorem~\ref{thm::conformal_pred_int} ensures that the prediction intervals produced by Algorithm~\ref{algo::IPCWsplit} have asymptotically valid coverage.

\begin{algorithm}[H] 
\begin{algorithmic}
\REQUIRE Data $\Dn = \{O_i: i \in \I\}$, miscoverage level $\alpha \in (0,1)$, regression algorithm $\A$ for the RMST, regression algorithm $\B$ for the censoring function $G$, split coefficient $\rho \in (0,1)$
\ENSURE Prediction interval, over $z \in \Rbb^d$
\STATE $\Gn = \B(\Dn)$ 
\STATE Randomly split $\I$ into subsets $\I_1,\I_2$ of sizes $n_1 = \lfloor \rho n 
\rfloor$ and $n_2 = n - n_1$ such that $\Dn[1] = \{O_i: i \in \I_1\}$ and $\Dn[2] = \{O_i: i \in \I_2\}$
\STATE $\muThat[1] = \A(\Dn[1])$
\STATE $R_i = | T_i \wedge \tau  - \muThat[1](Z_i) |$ and $\whi = \frac{\indicator{T_i \leq \tau} \delta_i}{1-\Gn(T_i- \mid Z_i)} + \frac{\indicator{T_i > \tau}}{1-\Gn(\tau \mid Z_i)}, i \in \I_2$ 
\STATE $\qhat = \inf\{t : \RniiG(t) \geq 1-\alpha\}$ the $(1-\alpha)$-quantile of the empirical cumulative distribution function of the residuals defined for all $t \in \Rbb$ by \[\RniiG(t) = \frac{1}{\sum_{i \in \I_2}\whi} \sum_{i \in \I_2} \indicator{R_i \leq t} \whi\]
\RETURN $\Csplit(z) = [ \muThat[1](z) - \qhat, \muThat[1](z) + \qhat]$ for all $z \in \Rbb^d$
\end{algorithmic}
\caption{IPCW Split Conformal prediction}
\label{algo::IPCWsplit}
\end{algorithm}

\begin{thm} \label{thm::conformal_pred_int}
Suppose that $\Gn$ is strongly consistent in the sense defined by Equation~\eqref{eq::strongconsistency}. If the observations $O_i$, $i\in \I$, are i.i.d., then for a new individual with independent event time $T^*$ and covariates $Z$ of the same law as the observations, under conditional independence (see Equation~\eqref{eq::conditionalindep}),
\[
    \lim_{n_2 \to \infty} \Prob\big(T^* \wedge \tau \in \Csplit(Z)\big) \geq 1 - \alpha
\]
for the split conformal prediction interval $\Csplit$ constructed by Algorithm~\ref{algo::IPCWsplit}. In addition, if the residuals have a continuous distribution, then 
\[
    \lim_{n_2 \to \infty} \Prob\big(T^* \wedge \tau \in \Csplit(Z)\big) = 1 - \alpha.
\]
\end{thm}

\begin{rk}
The validity of the results only relies on the accuracy of the estimate of the cumulative distribution function of the residuals, which depends on $n_2$. However, a higher value of $n_1$ yields a better predictor $\muThat[1]$, resulting in smaller residuals and shorter prediction intervals, which is desired in applications.
\end{rk}

\begin{rk} \label{rk::discontinuousres} 
    In most cases the residuals will have a continuous distribution. However, it may happen that the residuals' distribution puts a positive mass on discrete values when some predictions are identical and the corresponding observations $T_i \wedge \tau$ are also equal.
    This can occur if some individuals have the exact same covariate values or if they only differ for some covariates that are not used by the learning model. It is also common to use a reference model that does not take into account the covariates at all, in which case all predictions will be identical. When the residuals have a positive mass on some discrete values it may not be possible to exactly reach the desired confidence level and, in that case, the quantile of the residuals $\qhat$ will be chosen such that the coverage exceeds $1-\alpha$. Another approach consists in randomly breaking the ties in the residuals, as in~\citet{kuchibhotla_exchangeability_2021}.
\end{rk}

\subsection{IPCW In-Sample Split Conformal algorithm} \label{sec::roo}

The split algorithm introduced in the previous section aims at producing a prediction interval with the correct coverage for a new individual independent from the training set. However it might also be of interest to construct prediction intervals for the individuals in the training set itself. A simple, yet computationally inefficient, way to compute the prediction interval $\mathcal{C}(Z_i)$ for each $i\in \I$, is to apply Algorithm~\ref{algo::IPCWROO} using all observations but $O_i$ as training set. A more interesting approach is the Rank-One-Out (ROO) Split Conformal algorithm introduced by~\citet{lei_distribution-free_2018} to construct valid prediction intervals for all training data without requiring much more calculation. Similarly to Section~\ref{sec::split}, we adapt the ROO algorithm to the right-censoring framework using IPCW. Our weighted leave-one-out method is presented in Algorithm~\ref{algo::IPCWROO}. Asymptotic guarantees as in Theorem~\ref{thm::conformal_pred_int} can be similarly derived for Algorithm~\ref{algo::IPCWROO} but have been omitted for the sake of  conciseness.

\begin{algorithm}[H] 
\begin{algorithmic}
\REQUIRE Data $\Dn = \{O_i: i \in \I\}$, miscoverage level $\alpha \in (0,1)$, regression algorithm $\A$ for the RMST, regression algorithm $\B$ for the censoring function $G$, split coefficient $\rho \in (0,1)$
\ENSURE Prediction intervals at each $Z_i$, $i \in \I$
\STATE $\Gn = \B(\Dn)$
\STATE Randomly split $\I$ into two equal-sized subsets $\I_1,\I_2$ of sizes $n_1 = \lfloor n/2 \rfloor$ and $n_2 = n - n_1$ such that $\Dn[1] = \{O_i: i \in \I_1\}$ and $\Dn[2] = \{O_i: i \in \I_2\}$
\FOR{$k \in \{1,2\}$}
\STATE $\muThat[k] = \A(\Dn[k])$
\STATE $R_i = | T_i \wedge \tau  - \muThat[k](Z_i) |$ and $\whi = \frac{\indicator{T_i \leq \tau} \delta_i}{1-\Gn(T_i- \mid Z_i)} + \frac{\indicator{T_i > \tau}}{1-\Gn(\tau \mid Z_i)}, i \not\in \I_k$ 
\FOR{$i \not\in \I_k$}
\STATE $\hat{q}_{n_k,i} = \inf\{t : \hat{\R}_{n_k,i}^{\hat{G}}(t) \geq 1-\alpha\}$ the $(1-\alpha)$-quantile of the empirical cumulative distribution function of the residuals defined for all $t \in \Rbb$ by \[\hat{\R}_{n_k,i}^{\hat{G}}(t) = \frac{1}{\sum_{j \ne i, j \not\in \I_k}\whj} \sum_{j \ne i, j \not\in \I_k} \indicator{R_j \leq t} \whj\]
\STATE $\Croo(Z_i) = [ \muThat[k](Z_i) - \hat{q}_{n_k,i}, \muThat[k](Z_i) + \hat{q}_{n_k,i}]$
\ENDFOR
\ENDFOR
\RETURN $\Croo(Z_i)$ for all $i \in \I$
\end{algorithmic}
\caption{IPCW Rank-One-Out Split Conformal algorithm}
\label{algo::IPCWROO}
\end{algorithm}


\section{Variable importance} \label{sec::varimportance}

Conformal inference can also be used to assess the importance of each variable in the learning model. In particular, the Leave-One-Covariate-Out (LOCO) inference method described by \citet{lei_distribution-free_2018} can be adapted to the estimation of the RMST with censored data based on the new algorithms described in Section~\ref{sec::predintervals}.
To evaluate the importance of the $k$th variable, $k \in \{1,\dots,d\}$, the approach involves comparing the accuracy of the predictor trained with or without the $k$th variable. The magnitude of the performance difference indicates the variable significance in the model. Specifically, if we denote $\muThat[1]^{(-k)}$ the predictor trained on data $\Dn[1]$ without the $k$th variable, we are interested in the random variable
\begin{equation} \label{eq::deltak}
    \Delta_k(Z,T^*) = |T^* \wedge \tau- \muThat[1]^{(-k)}(Z)| - |T^* \wedge \tau - \muThat[1](Z)|
\end{equation} 
that measures the increase in prediction accuracy resulting from the inclusion of the $k$th variable in the model. The higher its value above zero, the greater the variable importance. This analysis can be conducted globally to assess the variable's overall importance in the model or locally to identify specific variable values that have a more significant impact on the outcome.


\subsection{Local measure of variable importance} \label{sec::vimpl}

The aim is to construct a prediction interval for $\Delta_k(Z,T^*)$. Let $\C_n$ be a prediction interval constructed with the split procedure for $T^*$ given $Z$ with coverage $1-\alpha$. For all $k = 1,\dots,d$, define
$$
    W_{n,k}(z) = \{ |t \wedge \tau - \muThat[1]^{(-k)}(z)| - |t \wedge \tau - \muThat[1](z)| : t \in \C_n(z) \}.
$$
Then from the asymptotic coverage of the prediction interval $\C_n$, we have
\begin{equation} \label{eq::Wnk}
    \Prob(\Delta_k(Z,T^*) \in W_{n,k}(Z), \text{ for all } k = 1,\dots,d) \xrightarrow[n \rightarrow \infty]{} 1 - \alpha.
\end{equation}   
It should be stressed that the result is not given conditionally on $Z=z$, however varying the value of the covariate will nevertheless have an impact on the intervals. Therefore, it will be of interest to evaluate the effect of the $k$th variable by plotting the intervals $W_{n,k}(Z_j)$ for $j = 1, \dots, m$.


\subsection{Global measure of variable importance} \label{sec::vimpg}

In this section, we consider that the data set $\Dn[1]$ is fixed. As opposed to the previous sections, we make the additional assumption that the censoring $C$ is independent from the time to event $T^*$ and from the covariates $Z$. We use the IPCW technique to construct the statistical test below, specifying that the weights use the Kaplan-Meier estimator for the censoring distribution based on the sample $\Dn[2]$. Formally, with $1-\Gn[2]$ denoting the Kaplan-Meier estimator of the censoring survival function, the weights $\whi$, $i \in \I_2$, become
\begin{equation} \label{eq::weights2}
    \whi = \frac{\indicator{T_i \leq \tau} \delta_i}{1-\Gn[2](T_i-)} + \frac{\indicator{T_i > \tau}}{1-\Gn[2](\tau)}.
\end{equation}

The global measure of the importance of the $k$th variable is constructed from the distribution of $\Delta_k(Z,T^*)$ marginally over $(Z,T^*)$. We consider the cumulative distribution function of $\Delta_k(Z,T^*)$ conditional on $T^*$ being lower than the threshold $\tau$, with $\tau < \tauH $. For simplicity, we suppose this distribution to be continuous and we are interested in inferring its median:
\begin{equation*}
    m_k = \text{median}\left[\Delta_k(Z,T^*) \mid T^* \leq \tau \right].
\end{equation*}
Specifically, we want to perform the following test
\[
    H_0 : m_k \leq 0 \text{ versus } H_1 : m_k > 0,
\]
which is equivalent to the test
\[
    H_0 : p_k \leq 1/2 \text{ versus } H_1 : p_k > 1/2.
\]
with $p_k = \Prob(\Delta_k(Z,T^*) \geq 0 \mid T^* \leq \tau)$.
Let $\Phi_k(t,z) = \indicator{|t - \muThat[1]^{(-k)}(z)| - |t - \muThat[1](z)| \geq 0, 0 \leq t \leq \tau}$. In particular, note that 
\[p_k  = \frac{1}{1-S(\tau)} \int \Phi_k(u,z) dF(u,z).\] 

We introduce the new test statistic 
\begin{equation} \label{eq::vimpgstatistic}
    T_{n_2} = \sqrt{\frac{n_2}{\hat{\sigma}^2(\Phi_k)}}\left(\frac{1}{1-\Sn[2](\tau)}\frac{1}{n_2} \sum_{i \in \I_2} \Phi_k(T_i,Z_i) \whi - \frac{1}{2} \right),
\end{equation}
where $\Sn[2]$ denotes the Kaplan-Meier estimator of the survival function, the censoring weights $\whi$, $i \in \I_2$ are defined as in Equation~\eqref{eq::weights2}, and
\begin{align*}
    \hat{\sigma}^2(\Phi_k) & = \frac{(1-\Sn[2](\tau))^{-2}}{n_2} \sum_{i \in \I_2} \indicator{T_i \leq \tau} \Bigg( \Phi_k(T_i,Z_i)\whi \\
    & \quad - \frac{\delta_i}{\hat{Y}_{n_2}(T_i)} \frac{1}{n_2} \sum_{j \in \I_2}  \left( \indicator{T_i \leq T_j } - \frac{\Sn[2](\tau)}{1-\Sn[2](\tau)} \right) \Phi_k(T_j,Z_j) \whj \Bigg)^2,
\end{align*}
with $\hat{Y}_{n_2}(t) = \frac{1}{n_2} \sum_{i \in \I_2} \indicator{T_i \geq t}$. This is a modification of the standard sign test statistic suited for right-censored data. In Theorem~\ref{thm::vimpG}, we show that this test statistic follows asymptotically a Gaussian distribution with variance equal to $1$ under $H_0$. This allows us to construct the statistical test with predefined level. The construction of an asymptotic confidence interval for $p_k$ is also proposed in Theorem~\ref{thm::vimpG}. 

\begin{thm} \label{thm::vimpG}
Let $\Dn[1]$ be a fixed data set. Let $\Sn[2]$, $1-\Gn[2]$ denote the Kaplan-Meier estimators of the functions $S$, $1-G$ respectively, computed on a data set $\Dn[2]$ independent from $\Dn[1]$. Let the censoring be independent from the time to event and from the covariates. We consider the weights $\whi$ as defined in Equation~\eqref{eq::weights2}. Then, under $H_0$, the test statistic $T_{n_2}$ defined in Equation~\eqref{eq::vimpgstatistic} follows asymptotically a Gaussian distribution with variance equal to $1$. Let $q_{1-\alpha}^{\N(0,1)}$ denote the $1-\alpha$ quantile of this distribution, then
\begin{equation*}
    \lim_{n_2 \to \infty} \Prob_{H_0}\left(T_{n_2} > q_{1-\alpha}^{\N(0,1)} \right) \geq \alpha.
\end{equation*}
We also have
\begin{equation} \label{eq::pkinterval}
    \lim_{n_2 \to \infty} \Prob\left(p_k \in \left[\frac{1}{1-\Sn[2](\tau)}\frac{1}{n_2} \sum_{i \in \I_2} \Phi_k(T_i,Z_i) \whi \pm \sqrt{\frac{\hat{\sigma}^2(\Phi_k)}{n_2}}q_{1-\alpha/2}^{\N(0,1)}\right] \right) = 1 - \alpha.
\end{equation}
\end{thm}

\begin{rk}
To simplify the proof, Theorem~\ref{thm::vimpG} is presented with the censoring estimator being calculated on $\Dn[2]$. We were not able to prove that the results would still hold if the censoring estimator was calculated from all the data but our simulations suggest that this would be the case.
\end{rk}

\begin{rk} \label{rk::pkjitt}
    In some cases, the inclusion or exclusion of the $k$th variable in the learning model may not affect the predictions. This may occur for models that are independent of covariates such as the Kaplan-Meier model. 
    When this happens, $\Delta_k(Z,T^*)$ equals $0$ and $p_k$ is equal to $1$. The test for global variable importance is not suited to these degenerated cases.
\end{rk}


\section{Simulations} \label{sec::simulations}

This section illustrates all the methods and results described in Sections~\ref{sec::mse},~\ref{sec::predintervals} and~\ref{sec::varimportance} through simulated experiments. We focus on cases where the censoring model is correctly specified, to align with our assumptions on the estimator of the censoring survival function. Additional simulations, where the model for the censoring survival function is misspecified, are available in the Supplementary Material. In Section \ref{sec::realdata}, we also present an application to the German Breast Cancer Study Group (GBCSG). In both the simulation and real data sections, the following learning algorithms are implemented and then evaluated using our methods.
\begin{description}
    \item[Integrated Kaplan-Meier:] A Kaplan-Meier estimator is fitted to $\Dn$ to estimate the survival function. By integrating this curve on the interval $[0,\tau]$, we obtain an estimation of the RMST that is identical for all individuals in the test sample $\Dm$. This represents a naive algorithm since it does not take into account the covariates.
    \item[Integrated Cox:] A Cox model \citep[see][]{cox_partial_1975} is fitted to $D_n$. Unless mentioned otherwise, no interaction between covariates is included in the model. The fitted model provides an estimation of the survival curve for each observation $O_j$ in $\Dm$, adjusted with respect to the covariate $Z_j$. The estimation of the RMST is obtained by integrating this curve with respect to time on the interval $[0,\tau]$.
    \item[Integrated RSF:] The procedure is identical to the one described above with the difference that an RSF \citep[see][]{ishwaran_random_2008} is fitted to $\Dn$ instead of a Cox model. 
    \item[Pseudo-observations and linear model:] We transform the censored data in $\Dn$ into pseudo-observations \citep[see][]{andersen_regression_2004}. We use the linear model as the link function to obtain a linear model for the RMST. Unless mentioned otherwise, no interaction term is included in the linear model.
\end{description}
We stress that, in the simulations, the RMST estimations obtained with those algorithms will never be truncated even if they exceed the interval $[0,\tau]$. This avoids creating ties in the residuals as mentioned in Remark~\ref{rk::discontinuousres}. Nevertheless, it is possible in practice to truncate both predictions and prediction intervals to avoid exceeding the interval $[0,\tau]$, in which case the coverage of the intervals will remain valid. We consider three different simulation schemes.
\begin{description}
    \item[Scheme A:] Following the setting from~\citet{wang_modeling_2018}, the event times are simulated according to the following linear model:
    $$
        T_i^* = \tilde{\beta}_0^{\top} Z_i  + \varepsilon_i,
    $$
    where $\tilde{\beta}_0 = (5.5,2.5,2.5)^{\top}$, the covariates are denoted $Z_i = (1,Z_i^1,Z_i^2)^{\top}$ with $Z_i^1,Z_i^2 \sim \mathcal{B}(0.5)$ and $\varepsilon_i \sim U[-3,3]$ is a random noise. From this model we obtain the following closed form for the RMST:
    \begin{equation} \label{eq::closeRMST}
        \muT^*(Z) = \Esp[T^* \wedge \tau \mid Z] = \beta_{00} + \beta_{01} Z^1 (1-Z^2) + \beta_{10} Z^2 (1-Z^1) + \beta_{11} Z^1 Z^2,
    \end{equation}
    where $\beta_0 = (\beta_{00},\beta_{01},\beta_{10},\beta_{11})^{\top}$ is computed from $10$ million Monte Carlo samples. The value of $\tau$ is fixed to $8.8$ which yields $\beta_0 = (5.5,2.097,2.097,3.16)^{\top}$. 
    Next, two different types of censoring are considered. In scheme \textbf{A1}, the censoring is simulated independently from the covariates according to an exponential law with parameter $\lambda = 0.07$, leading to $42 \%$ of censored data. In scheme \textbf{A2}, the censoring is simulated from a Cox model $\lambda(t\mid Z)=\lambda_0(t)\exp(\beta_1 Z^1+\beta_2 Z^2)$ with Weibull baseline hazard $\mathcal{W}(\nu,\kappa)$ defined as
    \begin{align*}
        \lambda_0(t)=\frac{\nu}{\kappa}\left(\frac{t}{\kappa}\right)^{\nu-1}.
    \end{align*} 
    We set $\kappa = 12$, $\nu = 6$, and $\beta_1 = 2$, $\beta_2 = 1$, leading to $44 \%$ of censored data. 
    \item[Scheme B:] The event times are simulated according to a Cox model with Weibull baseline hazard $\mathcal{W}(\nu,\kappa)$ and three covariates $Z = (Z^1,Z^2,Z^3)^{\top}$, where $Z^k \sim U[-a,a]$ for $k = 1,2,3$, with Cox regression parameters $\beta = (\beta_1,\beta_2,\beta_3)^{\top}$. Note that the survival function can be expressed as
    \[
        S(t\mid Z) = \exp\left[-\left(\frac{t}{\kappa}\right)^\nu\exp(\beta^{\top} Z)\right], 
    \]
    and the RMST can be obtained from Equation~\eqref{eq::RMSTdef}.
    Parameters are set to $\kappa = 2$, $\nu = 6$, $a = 5$, $(\beta_1, \beta_2, \beta_3) = (2,1,0)$. The censoring is simulated independently according to an exponential law with parameter $\lambda = 0.3$, leading to $47 \%$ censored data. The time horizon $\tau$ is chosen as the $90$th percentile of the observed times $T$ which gives $\tau=3.6$.
    \item[Scheme C:] Similarly to scheme B, the event times are simulated according to a Cox model with Weibull baseline hazard $\mathcal{W}(\nu,\kappa)$, $\lambda(t\mid Z)=\lambda_0(t)\exp(g(Z))$, with 
    \begin{align*}
        g(Z) & =  Z^3 - 3 Z^5 + 2 Z^1 Z^{10} + 4 Z^2 Z^7 + 3 Z^4 Z^5 - 5 Z^6 Z^{10} \\
        & \quad+ 3 Z^8 Z^9 + Z^1 Z^4 - 2 Z^6 Z^9 - 4 Z^3 Z^4 - Z^7 Z^8,
    \end{align*}
    and $\kappa = 2$, $\nu = 6$. Let $Z = (Z^1,\dots,Z^{15})$, we simulate the covariates such that $Z^j \sim \mathcal B(0.4)$ for $j \in \{2,4,6,9,11,12\} $ and $Z^j \sim U[0,1], ~ j \in \{1,3,5,7,8,10,13,14,15\}$. As a result, only the first $10$ covariates are associated with the event times, but the other $5$ covariates (that are non-informative) will still be included in our regression models. The survival function is expressed as:
    \begin{align*}
        S(t\mid Z) & = \exp\left[-\left(\frac{t}{k}\right)^\nu\exp(g(Z))\right],
    \end{align*}
     and the RMST can be obtained from Equation~\eqref{eq::RMSTdef}. The censoring distribution is the same as in scheme B, leading to $47 \%$ censored data. The time horizon $\tau$ is chosen as the $90$th percentile of the observed times $T$ which gives $\tau=2.8$.
\end{description}

\subsection{Illustration of the WRSS estimator}

In this section, we want to illustrate Theorem~\ref{thm::WRSS}. It states that, if the censoring estimator is consistent, then our prediction performance criterion, called WRSS (see Equation~\eqref{eq::WRSSdef}), converges towards the MSE as the sample size goes to infinity. We recall that $\text{MSE}(\muTtilde) = \Esp\left[(T^* \wedge \tau-\muTtilde(Z))^2 \right]$ and that $\muTtilde$ is the limit of the predictor, as defined in Equation~\eqref{eq::convergence}. We also recall that $\text{MSE}(\muTtilde)$ can be decomposed into an inseparability and imprecision terms, as shown in Equation~\eqref{eq::MSEdecomposition}. We will consider the simulation schemes \textbf{A1}, \textbf{A2} and two learning models. The first one is a linear model that is directly fitted on the minimum between the true event times and the time horizon $\tau$, using the correct link function  (see Equation~\eqref{eq::closeRMST}). It is considered as the oracle model. The second one is based on pseudo-observations with linear link function. The latter is implemented without interaction terms, i.e. only the covariates $Z^1$ ,$Z^2$ are included. 

In the scenario \textbf{A1}, the censoring is independent of the covariates and we use the Kaplan-Meier estimator to estimate its cumulative distribution function which is a consistent estimator (in the weak sense). Thus, the WRSS should converge to the MSE (see Theorem~\ref{thm::WRSS}). In addition, for the oracle model, the imprecision term should vanish as the model is correctly specified and the WRSS should therefore converge to the inseparability term. On the other hand, for the model based on pseudo-observations, the imprecision term will not vanish and the WRSS should then be larger than with the oracle model. Since the RMST has an explicit form, the inseparability term can be easily computed using Monte Carlo simulations. For the imprecision term of the model based on pseudo-observations, $\muTtilde$ was approximated by a predictor $\muThat$ trained on a sample of size $20,000$ and the expectation was calculated using a million Monte Carlo simulations. In Figure~\ref{fig::WRSS_indep}, we represent the WRSS based on train and test samples of equal size $100$, $500$ and $1,000$ for those two learning algorithms. The boxplots are obtained from $1,000$ repetitions. We clearly see that the oracle estimator converges towards the inseparability term, displayed in red in the figure. On the contrary, we see that, with the estimator based on pseudo-observations, the WRSS converges towards a value greater than the inseparability term, as expected. With the pseudo-observations model, we observe that the imprecision term is relatively small compared to the inseparability term which suggests that, while the regression model is incorrect, it still provides predictions that are close to the ones obtained using the oracle model. 

\begin{figure}[!ht]
    \centering
    \includegraphics[scale=0.9]{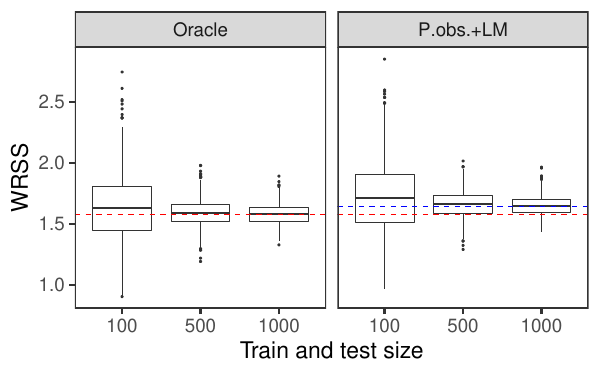}
    \caption{Distribution of $1,000$ replications of the WRSS estimator in the scenario \textbf{A1} and illustration of its convergence towards $\text{MSE}(\muTtilde)$ (see Equation~\eqref{eq::MSEdecomposition}), where $\muTtilde$ represents the limit defined in Equation~\eqref{eq::convergence}. Two learning models are compared. On the left panel, the oracle model~\eqref{eq::closeRMST} is a linear model fitted on the minimum between the true event times and $\tau$, using the correct link function. On the right panel, a linear model is implemented based on pseudo-observations, including all covariates without interaction terms. The red dotted line illustrates the inseparability term. It also represents the $\text{MSE}(\muTtilde)$ for the oracle model, whose imprecision term is null. The blue dotted line represents the $\text{MSE}(\muTtilde)$ for the model based on pseudo-observations, whose imprecision term is non-zero.}
    \label{fig::WRSS_indep}
\end{figure}

In the scenario \textbf{A2}, the censoring depends on the covariates through a Cox model. In this setting we also compare the performance of the two learning algorithms using different censoring models: a Kaplan-Meier method, a Cox model and an RSF model are implemented. We stress that the Kaplan-Meier method is no longer consistent while the Cox model is a consistent estimator for the censoring distribution. The results are illustrated in Figure~\ref{fig::WRSS_dep} where the boxplots are also obtained from $1,000$ repetitions. We clearly observe that the Kaplan-Meier method for the censoring distribution provides biased estimates of the MSE. On the other hand, the Cox and RSF models provide accurate estimations of the MSE. As in the scenario \textbf{A1}, the imprecision term is seen to vanish with the oracle model, and has a relatively small value compared to the inseparability term when using the pseudo-observations model. Those results also suggest that the RSF model for the censoring distribution is a consistent estimator (in the weak sense) since the results are almost identical to the results from the Cox model.

\begin{figure}[!ht]
    \centering
    \includegraphics[scale=0.9]{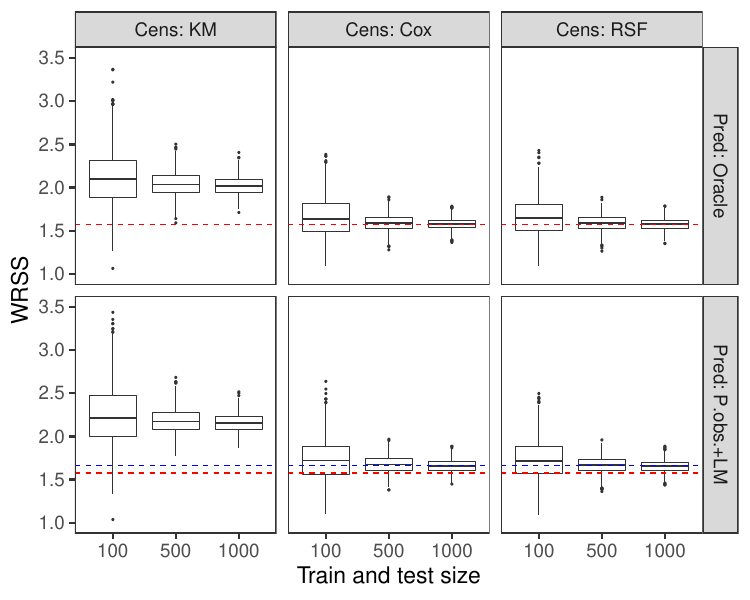}
    \caption{Distribution of $1,000$ replications of the WRSS estimator in the scenario \textbf{A2} and illustration of its convergence towards $\text{MSE}(\muTtilde)$ (see Equation~\eqref{eq::MSEdecomposition}), where $\muTtilde$ represents the limit defined in Equation~\eqref{eq::convergence}. Two learning models are compared. On the top panel, the oracle model~\eqref{eq::closeRMST} is a linear model fitted on the minimum between the true event times and $\tau$, using the correct link function. On the bottom panel, a linear model is implemented based on pseudo-observations, including all covariates without interaction terms. In addition, three censoring estimators are compared. From left to right, a Kaplan-Meier method, a Cox model and an RSF model. The red dotted line illustrates the inseparability term. It also represents the $\text{MSE}(\muTtilde)$ for the oracle model, whose imprecision term is null. The blue dotted line represents the $\text{MSE}(\muTtilde)$ for the model based on pseudo-observations, whose imprecision term is non-zero.}
    \label{fig::WRSS_dep}
\end{figure}

\subsection{Illustration of the IPCW Split Conformal algorithm}

In this section, we illustrate the construction of prediction intervals using the IPCW split conformal algorithm introduced in Section~\ref{sec::split}  (see also Algorithm~\ref{algo::IPCWsplit}). We simulate data using the simulation scheme \textbf{B} and we train all four learning models introduced at the beginning of Section~\ref{sec::simulations}. We first start by displaying  in Figure~\ref{fig::split_all} the prediction intervals at level $1-\alpha = 0.9$ on a sample of 10 individuals while the algorithms were trained on a single independent sample of size $4,000$.  As previously mentioned, intervals were not truncated. In order to have intervals comprised between $0$ and $\tau$, we could truncate both predictions and intervals, which would lead to a higher coverage (results not shown).

\begin{figure}[!ht]
    \centering
    \includegraphics[scale=0.9]{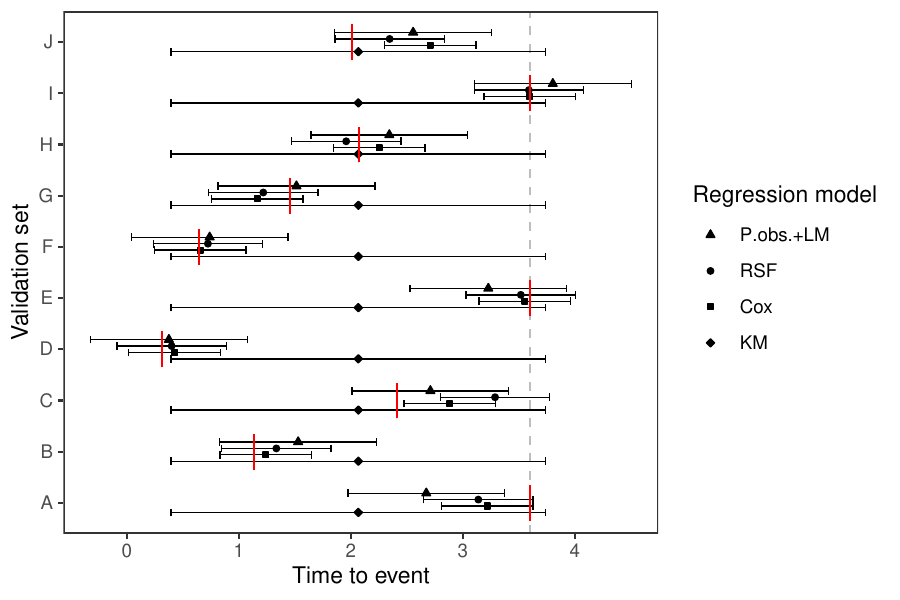}
    \caption{Prediction intervals at the $90\%$ level constructed with Algorithm \ref{algo::IPCWsplit} for four learning models: the Kaplan-Meier estimator, the Cox model, the RSF model and the linear model based on pseudo-observations. The training size is $n=4,000$ and the prediction intervals are constructed for $10$ individuals independent from the test set. All data are simulated according to the scenario \textbf{B}. The grey dotted line represents the time horizon $\tau = 3.6$. The red segments are placed at the minimum between the true event times of the test set and $\tau$.}
    \label{fig::split_all}
\end{figure}

Then, the coverage of the intervals, as claimed by Theorem~\ref{thm::conformal_pred_int}, is assessed in Figure \ref{fig::coverage_fig}, with $1-\alpha$ equal to $0.8$, $0.9$ or $0.95$. This time, the learning algorithms were trained on samples of size $n=300$, $500$ and $750$ where $n_1$ is fixed to $250$ and $n_2$ takes successively the values $50$, $250$ and $500$. The testing set, on which the empirical coverage is assessed, is of size $m=500$. The simulations were repeated $400$ times. We observe that for all learning models and confidence levels, the empirical coverage converges to $1-\alpha$, except for the integrated Kaplan-Meier algorithm which converges to a level greater than $1-\alpha$. Since this algorithm gives an RMST estimation that is identical for all individuals (as it does not take into account the covariates), the observations $T_i$ that are greater than $\tau$ will all have the same residual value. These ties make the distribution of the residuals discrete and therefore the empirical quantile of order $1-\alpha$ is such that the coverage exceeds $1-\alpha$ (see~Remark~\ref{rk::discontinuousres}).

\begin{figure}[!ht]
    \centering
    \includegraphics[scale=0.84]{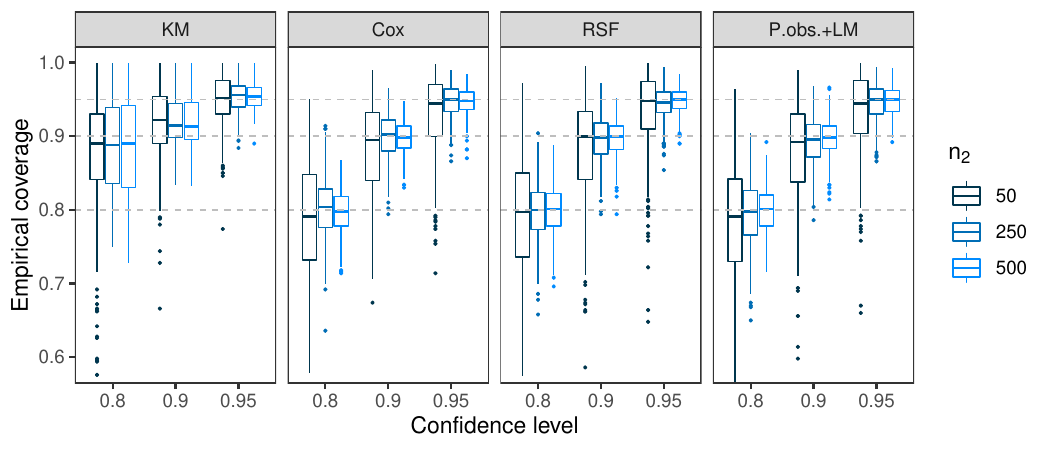}
    \caption{Empirical coverage for the prediction intervals constructed with Algorithm~\ref{algo::IPCWsplit} for four learning models: the Kaplan-Meier estimator, the Cox model, the RSF model and the linear model based on pseudo-observations. All data were simulated according to the scenario \textbf{B}.}
    \label{fig::coverage_fig}
\end{figure}

\subsection{Illustration of the LOCO variable importance measures} 

In this section, we provide illustrations of the use and performance of the LOCO variable importance measures introduced in Section~\ref{sec::varimportance}. 
We consider the simulation scheme \textbf{B}. In this scenario, three variables are considered. Only the first two are used to generate event times according to a Cox model, while the third has no impact on the outcome. For all learning algorithms (except the Kaplan-Meier model which does not take covariates into account), we want to test $H_0 : p_k \leq 1/2$ versus $H_1 : p_k > 1/2$, for $k=1, 2, 3$.
However, the value of $p_k$ for each algorithm is unknown and in particular, we do not know in advance if $p_k \leq 0.5$ ($H_0$ is true) or if $p_k > 0.5$ ($H_1$ is true). Their values are thus approximated via Monte Carlo simulations. We first simulate a training set $\Dn[1]$ of size $n_1=500$ which remains unchanged throughout the whole simulations (note that Theorem~\ref{thm::vimpG} holds for a fixed $\Dn[1]$). Next, we train the learning algorithms on this data set, simulate $10^5$ pairs $(T^*,Z)$ and compute $p_k$ from the distribution of the corresponding $\Delta_k(Z,T^*)$. Table~\ref{tab::calibrationTable} shows the resulting values, indicating that, for each model, $H_0$ is true for variable 3 while $H_1$ is true for variables 1 and 2.

\begin{table}[!ht]
    \centering
    \begin{tabular}{|l|rrr|}
            \hline
            Learning model  &  $p_1$ & $p_2$ & $p_3$ \\
            \hline
            Cox  &  0.87 & 0.79 & 0.49 \\
            Random Survival Forest  &  0.82 & 0.71 & 0.44 \\
            Pseudo-observations and linear model  &  0.84 & 0.70 & 0.46 \\
            \hline
        \end{tabular}
    \caption{Values of $p_k$, $k \in \{1,2,3\}$ for a fixed sample $\Dn[1]$, generated with $n_1=500$ according to the scenario \textbf{B}, for three learning models: the Cox model, the RSF model and the linear model based on pseudo-observations.}
    \label{tab::calibrationTable}
\end{table}

In practice, global variable importance can be determined using Theorem~\ref{thm::vimpG}. In Figure~\ref{fig::vimpGB}, using Equation~\eqref{eq::pkinterval}, the confidence intervals at the $90\%$ level for $p_k,\,k=1,2,3$ for the three learning algorithms are displayed based on $\Dn[1]$ and a data set $\Dn[2]$ of size $n_2=500$ simulated independently according to scenario \textbf{B}. We observe that the Cox model, the RSF model and the linear model based on pseudo-observations seem to all agree that only the first two covariates are important. 

\begin{figure}
    \centering
    \includegraphics[scale=0.86]{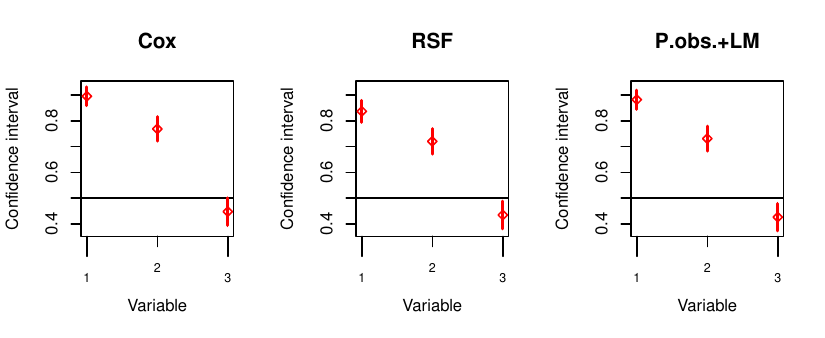}
    \caption{Confidence intervals at the $90\%$ level for $p_k,\,k=1,2,3$ (see Equation~\eqref{eq::pkinterval}), whose values are reported in Table~\ref{tab::calibrationTable}. The intervals were computed with the global variable importance measure applied to the fixed data set $\Dn[1]$ and a data set $\Dn[2]$ of size $n_2=500$ simulated independently according to scenario \textbf{B}. Three learning models are considered: the Cox model, the RSF model and the linear model based on pseudo-observations.}
    \label{fig::vimpGB}
\end{figure}

Using the same data sets $\Dn[1]$ and $\Dn[2]$, the local variable importance is illustrated for the three models and three covariates in Figure~\ref{fig::vimpLB}. 
All learning algorithms consider the first covariate to have importance in its high and low values and the third covariate to have no importance in the predictions, but they do not reach the same conclusions about the second covariate. The Cox model detects its importance for high and low values of the variable while the other two do not detect any importance of the variable. When the size of the data set increases, our simulations show that the RSF also comes to detect importance in high and low values (data not shown). The Cox and RSF models provide then very similar conclusions for all three covariates in terms of local variable importance.

\begin{figure}[!ht]
    \centering
    \includegraphics[scale=0.85]{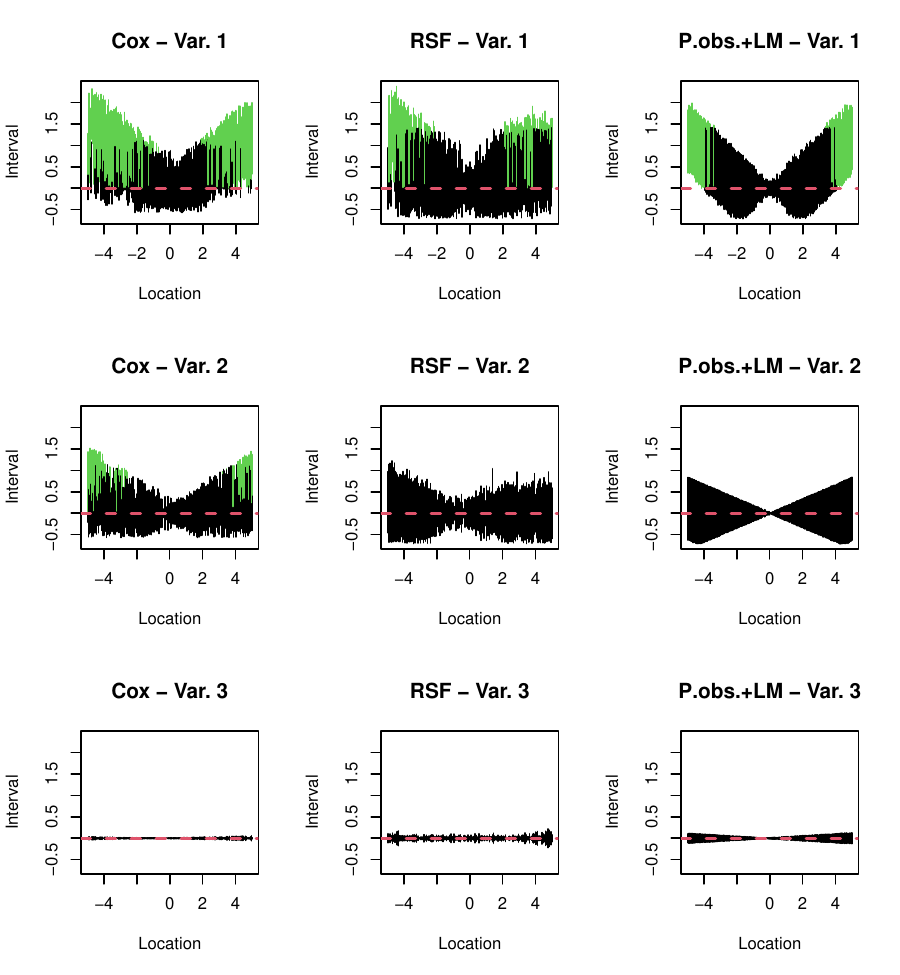}
    \caption{Local variable importance computed on the fixed data set $\Dn[1]$ and a data set $\Dn[2]$ of size $n_2=500$ simulated independently according to scenario the \textbf{B}, for three learning models: the Cox model, the RSF model and the linear model based on pseudo-observations. For the $k$th variable, the vertical segments show the intervals $W_{n,k}(Z_j)$, $j=1,\dots,m$, containing, with probability $1-\alpha=0.9$, the measure of the impact of $Z^k_j$ on predictions (see~\eqref{eq::Wnk}). Vertical segments lying above zero are colored in green.}
    \label{fig::vimpLB}
\end{figure}

Still using the same fixed data set $\Dn[1]$, we empirically assess the calibration and power of our test for global importance by simulating $1,000$ samples $\Dn[2]$ of size $n_2=500$ and by computing for each one the p-value for the statistical test.
The histograms of those p-values for each value of $k$ and all three algorithms are displayed in Figure~\ref{fig::calibrationPlot}. 
When $k=3$, we observe a skewed distribution of the p-values towards $1$ and $5\%$ rejection rates below $5\%$. This was expected since the $H_0$ hypothesis is composite and, according to Table~\ref{tab::calibrationTable}, $H_0$ is true for $k=3$ for all models. Finally, when $k=1, 2$ ($H_1$ is true) we observe that all three algorithms have a very strong power (all p-values are less than 0.01).

\begin{figure}[!ht]
    \centering

    \subfloat[$5\%$ rejection rates]{%
    \begin{tabular}{|l|rrr|}
            \hline
            Variable  & Cox & RSF & P.obs.+LM \\
            \hline
            1  & 1 & 1 & 1 \\
            2  & 1 & 1 & 1 \\
            3  & 0.027 & 0.001 & 0.004 \\
            \hline
    \end{tabular}
    \label{fig::reject_rate}
    }

    \subfloat[Distribution of the p-values]{
        \includegraphics[scale=0.9]{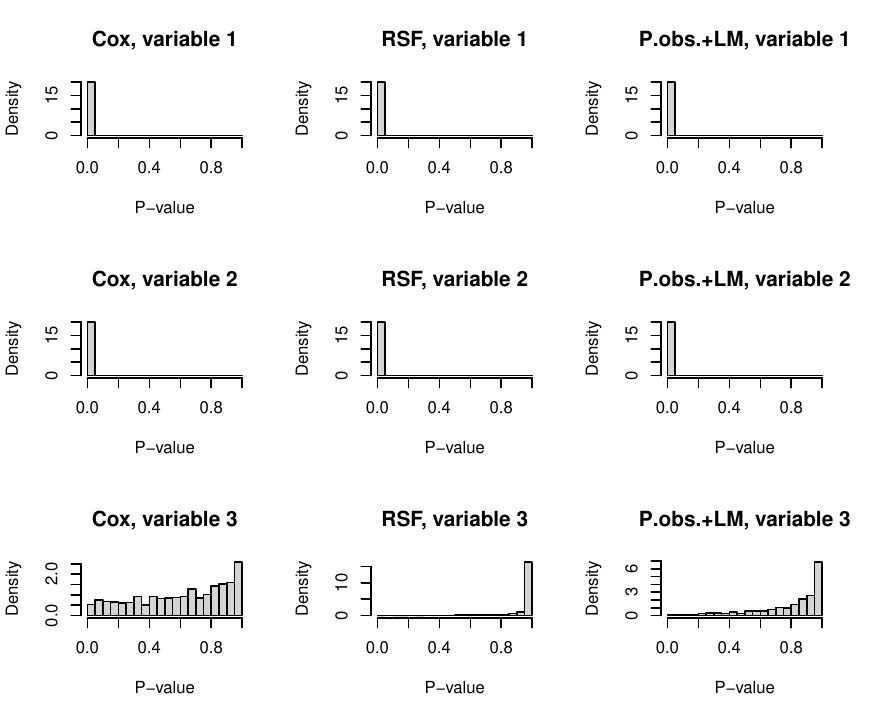} 
        \label{fig::calibration_plot}
    }

    \caption{Distribution of the p-values and $5\%$ rejection rates from $1,000$ repetitions of the LOCO global variable importance test, for three learning models: the Cox model, the RSF model and the linear model based on pseudo-observations. The sample $\Dn[1]$ was generated with $n_1=500$  and remained fixed while $\Dn[2]$ was simulated $1,000$ times with $n_2=500$ in order to obtain the distribution of the p-values. All data were simulated according to the scenario \textbf{B}. For all learning models, $H_0$ is true for variable 3 while $H_1$ is true for variables 1 and 2, see Table~\ref{tab::calibrationTable}.}
    \label{fig::calibrationPlot}
\end{figure}

\clearpage

\subsection{Multi-splitting}\label{sec::multisplitting}

We want to emphasize that variable importance results depend not only on the learning algorithm but also on the split of the data. In particular when it comes to assessing global variable importance, results may vary drastically depending on the split, for instance when the link between covariates and outcome is complex such as in the simulation scheme \textbf{C}. As an illustration, the LOCO global variable importance test is performed for all three learning models, on a data set of size $1,000$ simulated with the simulation scheme \textbf{C}. The oracle Cox model taking interactions into account is added for comparison. Figure~\ref{fig::vimpGC} presents the outputs of the procedure repeated twice, the only difference being the data split. Some variables seem to improve the accuracy of the model on one split, and on the contrary to degrade it on the other.

\begin{figure}[!ht]
    \centering
    \begin{subfigure}[b]{1\textwidth}
        \centering
        \includegraphics[width=0.9\textwidth]{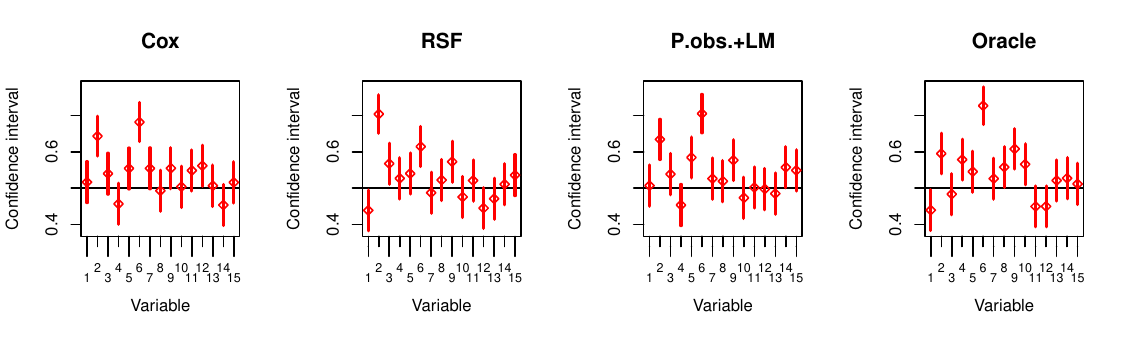}
        \caption{Split 1}    
    \end{subfigure}
    \vskip\baselineskip
    \begin{subfigure}[b]{1\textwidth}   
        \centering 
        \includegraphics[width=0.9\textwidth]{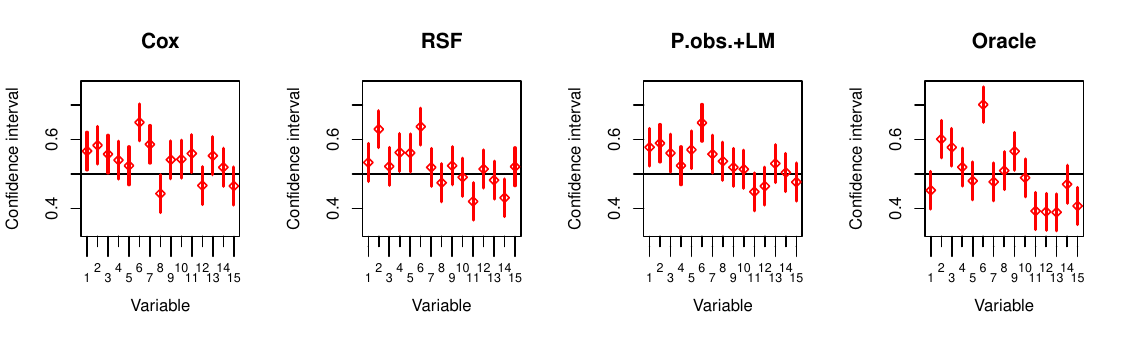}
        \caption{Split 2}     
    \end{subfigure}
    \caption{Confidence intervals at the $90\%$ level for $p_k,\,k=1,\dots,15$ (see Equation~\eqref{eq::pkinterval}), computed with the global variable importance measure on a data set of size $n=1,000$ simulated using the scenario \textbf{C}, for four learning models: the Cox model, the RSF model, the linear model based on pseudo-observations and and the oracle Cox model taking interactions into account. The procedure is conducted on two different splits of the data.} 
    \label{fig::vimpGC}
\end{figure}

To reduce the effect of the split, methods like multi-splitting described in \citet{diciccio_exact_2020} can be considered. Specifically the authors of this paper propose to choose $M$ splits of the data, and then compute a p-value for each split. Twice the resulting median or average p-value provides a valid p-value for the overall test. The overall Type 1 error can be controlled and bounded under $\alpha$, however this comes with a loss of power compared to a single split. 
We selected $M=50$ splits of our data set simulated according to simulation scheme \textbf{C}. For each split, each learning algorithm and each variable, we computed a p-value. Twice the median p-value is computed for each setting and reported in Table~\ref{tab::multisplittingC}. 
Binary variables such as $Z^2$, $Z^6$ and $Z^9$ seem to play a significant role on the predictions. There is a consensus on  variables $Z^{11}$ to $Z^{15}$, which are uninformative and indeed seem to have no impact. The effect of other variables is elusive, and can vary depending on the model.

\begin{table}[!ht]
    \centering

    \begin{tabular}{|l|ll|ll|ll|ll|}
\hline
Variable & \multicolumn{2}{l|}{P-value Cox} & \multicolumn{2}{l|}{P-value RSF} & \multicolumn{2}{l|}{P-value P.obs.+LM} & \multicolumn{2}{l|}{P-value Oracle}\\
\hline
1 & 0.62 &  & 1 &  & 0.588 &  & 1 & \\

2 & 0.007 & ** & 0 & *** & 0.012 & * & 0.001 & **\\

3 & 0.545 &  & 1 &  & 0.489 &  & 0.467 & \\

4 & 1 &  & 0.863 &  & 1 &  & 0.145 & \\

5 & 0.105 &  & 0.466 &  & 0.054 & . & 0.102 & \\

6 & 0 & *** & 0.001 & ** & 0 & *** & 0 & ***\\

7 & 0.453 &  & 0.54 &  & 0.804 &  & 0.225 & \\

8 & 0.958 &  & 1 &  & 0.677 &  & 0.576 & \\

9 & 0.087 & . & 0.05 & . & 0.067 & . & 0.046 & *\\

10 & 0.992 &  & 1 &  & 1 &  & 0.46 & \\

11 & 1 &  & 1 &  & 1 &  & 1 & \\

12 & 1 &  & 1 &  & 1 &  & 1 & \\

13 & 1 &  & 1 &  & 1 &  & 1 & \\

14 & 1 &  & 1 &  & 1 &  & 1 & \\

15 & 1 &  & 1 &  & 0.893 &  & 1 & \\
\hline
    \end{tabular}

    \smallskip
    
    Significance codes: 0 '***' 0.001 '**' 0.01 '*' 0.05 '.' 0.1 ' ' 1
    \caption{Multi-splitting on a data set of size $n=1,000$ simulated according to scenario the \textbf{C} for four learning models: the Cox model, the RSF model, the linear model based on pseudo-observations and the oracle Cox model taking interactions into account. $M=50$ splits of the data are randomly chosen. The global variable importance test is performed on each one. The resulting $M$ p-values are aggregated by twice the median value, serving as p-value of level $\alpha$ for the overall test.}
    \label{tab::multisplittingC}
\end{table}


\section{Application on real data} \label{sec::realdata}

We illustrate our comprehensive framework for evaluating RMST estimators on the classic German Breast Cancer Study Group (GBCSG) data set, which was first introduced in \citet{schumacher_randomized_1994} and is now available in the \texttt{survival} R package. The GBCSG gathered patients with node-positive breast cancer. The study was conducted from 1984 to 1989 and aimed to investigate the impact of hormone treatment on recurrence-free survival time. The event of interest is then the recurrence of a cancer, which was observed for 299 of the 686 patients. 246 patients were treated with additional hormonal therapy. Finally, prognostic factors were collected on all patients: age, menopausal status, tumor size, tumor grade, number of positive nodes, progesterone receptor and estrogen receptor. 
We estimated the censoring survival function using the Kaplan-Meier model.
We set the time horizon $\tau$ at the 90th percentile of the observed times distribution (2014 days) and were interested in the recurrence time prediction up to that time limit. 
We first investigated the possible dependence between censoring and covariates. To that aim we fitted a Cox model and an RSF on the censoring distribution. No evidence of dependence was highlighted. All p-values for the Cox model were over $0.149$, except for the treatment variable ($0.019$), but the global p-value was equal to $0.155$. As for the variable importance measure per permutation used in the RSF, it did not exceed $0.002$ for any of the variables, and $0.030$ for the treatment variable.
We tried predicting the recurrence time with all four of the learning algorithms described in Section~\ref{sec::simulations}. For each one, the WRSS is computed based on $20$-fold cross-validation in Figure~\ref{fig::mse_brcancer}. The Kaplan-Meier model is less performing than the covariate-dependent models, indicating that the chosen prognostic factors indeed play a role on the recurrence-free survival time. The RSF seems to have a slightly better performance than the Cox model and the linear model based on pseud-observations, though the variability of the results makes it difficult to identify clearly which algorithm is best suited to the data. Results of the global variable importance test using multi-splitting on $M = 40$ splits are displayed in Table~\ref{tab::multisplittingbrcancer}. Interestingly, no variables turn out to be important in the RSF predicting the recurrence time, despite its observed overall good predictive performance. Hormonal therapy, progesterone receptors, and tumor grade are important in the prediction with both the Cox model and the linear model based on pseudo-observations.

\begin{figure}[!ht]
    \centering
    \includegraphics[scale=0.73]{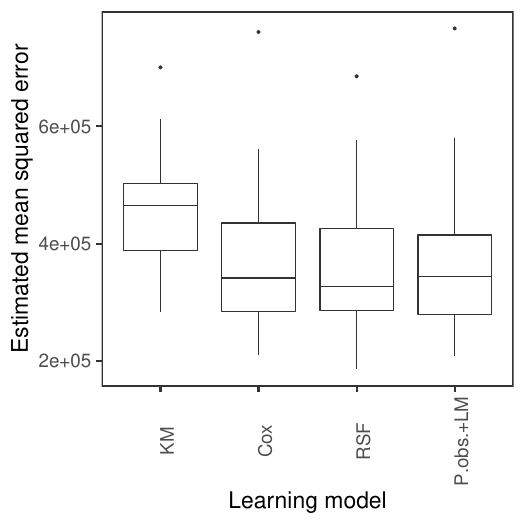}
    \caption{Estimation of the mean squared error on the GBCSG data set using the WRSS based on $20$-fold cross-validation, for four learning models: the Kaplan-Meier estimator, the Cox model, the RSF model and the linear model based on pseudo-observations.}
    \label{fig::mse_brcancer}
\end{figure}

\begin{table}[!ht]
    \centering

    \begin{tabular}{|l|ll|ll|ll|ll|}
\hline
Variable & \multicolumn{2}{l|}{P-value Cox} & \multicolumn{2}{l|}{P-value RSF} & \multicolumn{2}{l|}{P-value P.obs.+LM} \\
\hline
hormon & 0.001 & ** & 0.166 &  & 0.001 & **\\

age (years) & 1 &  & 1 &  & 1 & \\

menopausal status & 0.203 &  & 1 &  & 0.093 & .\\

tumour size (mm) & 0.989 &  & 0.498 &  & 1 & \\

number of positive nodes & 1 &  & 0.329 &  & 1 & \\

progesterone receptor (fmol) & 0 & *** & 0.359 &  & 0 & ***\\

estrogen receptor (fmol) & 1 &  & 1 &  & 1 & \\

tumour grade $\geq 2$ & 0 & *** & 0.516 &  & 0 & ***\\

\hline
\end{tabular}

    \smallskip
    
    Significance codes: 0 '***' 0.001 '**' 0.01 '*' 0.05 '.' 0.1 ' ' 1
    \caption{Multi-splitting on the GBCSG data set for three learning models: the Cox model, the RSF model and the linear model based on pseudo-observations. $M=40$ splits of the data are randomly chosen. The global variable importance test is performed on each split. The resulting $M$ p-values are aggregated by twice the median value, serving  as p-value of level $\alpha$ for the overall test.}
    \label{tab::multisplittingbrcancer}
\end{table}


\section{Conclusion}

One way of assessing the fit of a model and its predictive ability is to evaluate its MSE based on a test sample. This is a general measure that allows us to compare a wide range of models without making any specific model assumption. In the presence of right-censoring, due to tail estimation issues, it is natural to focus on the prediction of the restricted time $T^*\wedge \tau$ based on a clinically relevant fixed time horizon $\tau$ \citep[see][]{eaton_designing_2020}. Over the recent years, there has been a wide variety of new models that can handle censored data from the random survival forest method to neural networks based methods~\citep[see][]{ishwaran_random_2008,zhao_deep_2021}. Those methods make few assumptions as compared to the ubiquitous Cox model that relies on the proportional hazard assumption and it is interesting to evaluate and compare their predictive abilities. Based on the MSE it is well known that the best prediction model is the conditional expectation of the restricted time and therefore predicting the restricted time amounts to estimating the RMST.

Methods for evaluating RMST estimation models remain limited. 
In this context, we propose a complete framework for evaluating a time to event predictor using the RMST. Specifically, we propose methods to assess its predictive ability and variability, and to estimate variable importance, globally and locally. The steps of the implementation can be summarized as follows.
\begin{enumerate}
    \item Evaluate the predictive performance of the predictor quantitatively with the WRSS \eqref{eq::WRSSdef}, an estimator of the MSE based on IPCW. To reduce the variability induced by the splitting procedure, $V$-fold cross-validation is recommended. 
    \item Construct prediction intervals with the IPCW Conformal Algorithm \ref{algo::IPCWsplit} or \ref{algo::IPCWROO}, as illustrated in Figure~\ref{fig::split_all}. 
    \item For each variable of interest, evaluate its global importance in the regression model with the statistical test \eqref{eq::vimpgstatistic} and construct the corresponding confidence interval given by \eqref{eq::pkinterval}. Note that this measure is inherent in the model and in the data split and can differ from the true significance of the variable. 
    \item Compute the local variable importance when it is relevant, that is for variables with a sufficient amount of possible values. This measure is also inherent in the model and in the data split.
\end{enumerate}
In a context of high-dimensional data, one can form groups of variables that are relevant according to domain-specific knowledge. To this aim, in the definition~\eqref{eq::deltak}, the $k$ which represents the $k$th variable can equivalently represent the $k$th group of variables. Otherwise, we recommend focusing only on a subset of relevant covariates by first applying a pre-selection step on the data, typically with variable selection methods such as the Lasso or based on an AIC criterion \citep[see for instance][]{kojima_variable_2022}. These methods allow us to reduce the computational cost of our procedure.  

Under consistency assumptions defined in Section~\ref{sec::notations}, we proved that all the methods described above are asymptotically valid. 
Where standard conformal methods verify finite-sample validity, IPCW conformal algorithms are only valid asymptotically because of the estimation of the censoring survival function.
Nevertheless, as for standard conformal methods, no assumption is needed on the estimator of the RMST nor on the true data distribution.

We want to stress that the whole analysis procedure is model-agnostic, thus robust against model misspecification, but not against censoring model misspecification. 
Results from the Supplementary Material indicate that, when the censoring model is misspecified, there is a slight degradation in performance. The main impact is observed on the power of the global variable importance test, where important variables with weak effects are more difficult to detect.
However, estimating the distribution of the time to event (conditional on covariates) and estimating the distribution of the censoring (conditional on covariates) are inherently asymmetric tasks. In applications, censoring often exhibits minimal dependence on covariates, or may not depend on them at all. Consequently, estimating the censoring distribution is generally a simpler task compared to estimating the time to event distribution. It is important to note that this discussion is broad and applies to all IPCW methods. For instance, the same assumption on the censoring model is required for the standard Brier score~\citep[see][]{gerds_consistent_2006}.

Another limitation concerns the independence assumption between censoring and covariates for the global variable importance test, in Theorem~\ref{thm::vimpG}. Unfortunately, we were not able to extend this last result to a more general dependence relation. The distribution of the test statistic relies on properties of Kaplan-Meier integrals and new theoretical results would need to be derived for a test statistic that would be based on a general estimator of the censoring cumulative distribution. This is left to future research.

Split conformal prediction offers a computationally efficient approach to conducting distribution-free predictive inference in regression tasks. However, it relies on a single random split of the data, which can significantly influence the results.
Various methods exist to aggregate split conformal prediction intervals across multiple splits (cross-conformal prediction~\citep{vovk_cross-conformal_2015}, jackknife+ and K-fold CV+ prediction~\citep{barber_predictive_2021}, K-subsample conformal prediction~\citep{gupta_nested_2022} and multi-split conformal prediction based on Markov’s inequality~\citep{solari_multi_2022}).
These methods allow us to mitigate the impact of the split on the prediction intervals created with Algorithms~\ref{algo::IPCWsplit} and~\ref{algo::IPCWROO}. Developing a multi-split procedure for the global variable importance test is an avenue for future research.


\section{Technical proofs} \label{sec::proofs}

\subsection{Proofs for Section~\ref{sec::mse}}

\begin{proof}[Proof of Theorem~\ref{thm::WRSS}]
    The proof follows the outline of the proof of Theorem 5.2 in~\citet{gerds_consistent_2006}. It essentially relies on the assumptions of weak consistency of the censoring estimator defined by Equation~\eqref{eq::weakconsistency}, and on $\muThat \in \M_\A$ and its limit $\muTtilde \in \M_\A$ being bounded. 
\end{proof}


\subsection{Proofs for Section~\ref{sec::predintervals}}

Let $\{\wi, i \in \I_2\}$ be the inverse probability censoring weights computed with the unknown function $G$, i.e. for all $i$
\[
    \wi = \frac{\indicator{T_i \leq \tau} \delta_i}{1-G(T_i- \mid Z_i)} + \frac{\indicator{T_i > \tau}}{1-G(\tau \mid Z_i)}\cdot
\]

\begin{lemma} \label{lem::ipcw_cond_mean}
Let $(T,\delta,Z)$ be an observation independent of $\Dn[1]$. 
Let 
\begin{align*}
    f : \Rbb^+ \times \Rbb^d & \to \Rbb \\
    (t,z) & \mapsto f(t,z,\Dn[1]),
\end{align*}
such that $ \Esp[f(T^* \wedge \tau,Z,\Dn[1]) \mid  \Dn[1]]$ is almost surely finite. Then

\[
  \Esp[f(T \wedge \tau,Z,\Dn[1]) \w\mid \Dn[1]] = \Esp[f(T^* \wedge \tau,Z,\Dn[1]) \mid  \Dn[1]].
\]
\end{lemma}

\begin{proof}
We have:
\begin{align*}
    & \Esp[f(T \wedge \tau,Z,\Dn[1]) \w\mid \Dn[1]] \\
    & \quad = \Esp\left[ \indicator{T \leq \tau} \frac{f(T \wedge \tau,Z,\Dn[1]) \delta}{1-G(T- \mid Z)} + \indicator{T > \tau} \frac{f(T \wedge \tau,Z,\Dn[1])}{1-G(\tau \mid Z)} \Bigm\vert \Dn[1] \right] \\
    & \quad = \Esp\left[\indicator{T^* \leq \tau} \frac{f(T^*,Z,\Dn[1])}{1-G(T^*- \mid Z)} \Esp[\indicator{T^* < C} \mid T^*,Z,\Dn[1]] \Bigm\vert  \Dn[1] \right] \\
    & \qquad + \Esp\left[ \indicator{T^* > \tau} \frac{f(\tau,Z,\Dn[1])}{1-G(\tau \mid Z)} \Esp[\indicator{C > \tau} \mid T^*,Z,\Dn[1]] \Bigm\vert  \Dn[1] \right]\\
    & \quad = \Esp\left[\indicator{T^* \leq \tau} f(T^*,Z,\Dn[1]) + \indicator{T^* > \tau}f(\tau,Z,\Dn[1]) \mid  \Dn[1] \right] \\
    & \quad = \Esp[f(T^* \wedge \tau,Z,\Dn[1]) \mid  \Dn[1]],
\end{align*}
where in the third equality we used the independence assumption between $C$ and $T^*$ conditional on $(Z,\Dn[1])$.
\end{proof}

\begin{lemma} \label{lem::ipcw_prob_conv}
For $(T^*,Z)$ independent of $\Dn[1]$, let $f$ be a function verifying $\Esp[\lvert f(T^* \wedge \tau,Z,\Dn[1]) \rvert \mid  \Dn[1]] < \infty$ a.s. for all $\tau < \tauH$. If $\Gn$ is strongly consistent in the sense defined by \eqref{eq::strongconsistency} then
\[
  \frac{1}{n_2} \sum_{i \in \I_2} f(T_i \wedge \tau,Z_i,\Dn[1]) \whi \xrightarrow[n_2 \to \infty]{a.s.} \Esp[f(T^* \wedge \tau,Z,\Dn[1]) \mid  \Dn[1]].
\]
\end{lemma}

\begin{proof}
\begin{align*}
    \frac{1}{n_2} \sum_{i \in \I_2} f(T_i \wedge \tau,Z_i,\Dn[1]) \whi & = \frac{1}{n_2} \sum_{i \in \I_2} f(T_i \wedge \tau,Z_i,\Dn[1]) \wi \\
    & \quad + \frac{1}{n_2} \sum_{i \in \I_2} f(T_i \wedge \tau,Z_i,\Dn[1]) (\whi -  \wi) 
\end{align*}
As the $f(T_i \wedge \tau,Z_i,\Dn[1])$, $i \in \I_2$, are conditionally independent given $\Dn[1]$ with the same conditional distribution, then by the strong conditional law of large numbers~\citep[see Theorem 4.2 in][]{majerek_conditional_2005} and by Lemma~\ref{lem::ipcw_cond_mean},
\[
    \frac{1}{n_2} \sum_{i \in \I_2} f(T_i \wedge \tau,Z_i,\Dn[1]) \wi \xrightarrow[n_2 \to \infty]{a.s.} \Esp[f(T^* \wedge \tau,Z,\Dn[1]) \mid  \Dn[1]].
\]
On the other hand, as that there exists $\epsilon > 0$ such that $1-\Gn(\tau \mid Z_i) > \epsilon$ for $n$ high enough and $\tau < \tauH$,

\begin{align*}
    & \bigg\lvert \frac{1}{n_2} \sum_{i \in \I_2} f(T_i \wedge \tau,Z_i,\Dn[1]) (\whi -  \wi) \bigg\rvert \\
    & \quad = \bigg\lvert \frac{1}{n_2} \sum_{i \in \I_2} f(T_i \wedge \tau,Z_i,\Dn[1]) \indicator{T_i \leq \tau} \delta_i \frac{G(T_i- \mid Z_i) - \Gn(T_i- \mid Z_i)}{(1-\Gn(T_i- \mid Z_i))(1-G(T_i- \mid Z_i))} \\
    & \qquad + \frac{1}{n_2} \sum_{i \in \I_2} f(T_i \wedge \tau,Z_i,\Dn[1]) \indicator{T_i > \tau} \frac{G(\tau \mid Z_i) - \Gn(\tau \mid Z_i)}{(1-\Gn(\tau \mid Z_i))(1-G(\tau \mid Z_i))} \bigg\rvert \\
    & \quad \leq \epsilon^{-1} \sup_{s \leq \tau, z \in \Rbb^d} \big| \Gn(s \mid z) - G(s \mid z) \big| \frac{1}{n_2} \sum_{i \in \I_2} \lvert f(T_i \wedge \tau,Z_i,\Dn[1]) \rvert \wi \\
    & \quad \xrightarrow[n_2 \to \infty]{a.s.} 0,
\end{align*}
where the convergence to zero follows from the strong consistency of the censoring estimator defined by~\eqref{eq::strongconsistency} and the law of large numbers for $ \lvert f(T_i \wedge \tau,Z_i,\Dn[1]) \rvert \wi$.
\end{proof}

\begin{proof}[Proof of Theorem~\ref{thm::conformal_pred_int}]

Applying Lemma~\ref{lem::ipcw_prob_conv} to the functions $(u,z,\Dn[1]) \mapsto 1$  and $(u,z,\Dn[1]) \mapsto \indicator{|u - \muThat[1](z)| \leq t}$ for $t \in \Rbb$ gives 
\[
  \frac{\sum_{i \in \I_2} \whi}{n_2}  \xrightarrow[n_2 \to \infty]{a.s.} 1
\]
and for all $t \in \Rbb$,
\[
\frac{1}{n_2} \sum_{i \in \I_2} \indicator{R_i \leq t} \whi \xrightarrow[n_2 \to \infty]{a.s.} \R^*_1(t):= \Prob(R^* \leq t \mid \Dn[1]).
\]

Then for all $t \in \Rbb$,
\[
    \RniiG(t) = \frac{1}{\sum_{i \in \I_2} \whi} \sum_{i \in \I_2} \indicator{R_i \leq t} \whi \xrightarrow[n_2 \to \infty]{a.s.} \R^*_1(t).
\]

For all $\alpha \in (0,1)$, let $\qhat^{\alpha} = \inf\{t, \RniiG(t) \geq 1-\alpha\}$ and $q^{\alpha} = \inf\{t, \R^*_1(t) \geq 1-\alpha\}$ be the $1-\alpha$ quantiles of the cumulative distribution functions $\RniiG$ and $\R^*_1$ respectively. 
Then from Lemma 21.2 in~\citet{van_der_vaart_asymptotic_1998} we have that for all $\alpha \in (0,1)$ 
\[
    \qhat^{\alpha} \xrightarrow[n_2 \to \infty]{a.s.} q^{\alpha}.
\]

Finally, applying the continuous mapping theorem to the function $\R^*_1$ gives
\[ \R^*_1(\qhat^{\alpha}) \xrightarrow[n_2 \to \infty]{a.s.}  \R^*_1(q^{\alpha}) \]
i.e.
\[\Prob(R^* \leq \qhat^{\alpha} \mid \Dn[1]) = \R^*_1(\qhat^{\alpha}) \xrightarrow[n_2 \to \infty]{a.s.}  \R^*_1(q^{\alpha}) = \Prob(R^* \leq q^{\alpha} \mid \Dn[1]) \geq 1 - \alpha. \]
In particular, if the residuals have a continuous distribution given $\Dn[1]$, then  $\Prob(R^* \leq q^{\alpha} \mid \Dn[1]) = 1-\alpha$.

\end{proof}


\subsection{Proofs for Section~\ref{sec::varimportance}}

We provide here a slightly more general result than the one given by Theorem~\ref{thm::vimpG}, where the function $\Phi_k$ can be any bounded function with support on $[0,\tau]$. For clarity, we state the theorem in its general form below.

\begin{prop} \label{prop::CLT}
    Let $\tau < \tauH$ and $\Phi_k(t,z): (t,z) \in \Rbb^+ \times \Rbb^d \to \Rbb^+$ be a uniformly bounded function with support $[0,\tau]$. Let 
    \[
        \theta_k  = \frac{1}{1-S(\tau)} \intT \intZ \Phi_k(u,z) dF(u,z).
    \] 
    Suppose that $C\indep (T^*,Z)$.
    Let $\Sn[2]=1-\Fn[2]$, and $1-\Gn[2]$ denote the Kaplan-Meier estimators of the functions $S=1-F$ and $1-G$ computed on $\Dn[2]$. The censoring weights $\whi$, $i \in \I_2$ are defined as in Equation~\eqref{eq::weights2}. Let
    \[
        \hat{\theta}_k = \frac{1}{1-\Sn[2](\tau)}\intT \intZ \Phi_k(u,z) d\Fn[2](u,z) = \frac{1}{1-\Sn[2](\tau)}\frac{1}{n_2}\sum_{i \in \I_2} \Phi_k(T_i,Z_i) \whi
    \]
    be an estimator of $\theta_k$. Then
    \[
        \sqrt{n_2}\left(\hat{\theta}_k - \theta_k \right) \xrightarrow[n_2 \to \infty]{} \N(0,\sigma^2(\Phi_k)) \; \text{in distribution,}
    \]
    where
    \begin{align*}
        \sigma^2(\Phi_k) & = \frac{1}{(1-S(\tau))^2}  \intT \intZ  \Bigg( \frac{\Phi_k(u,z)}{1-G(u-)} - \frac{\Bar{\Phi}_k(u)}{1-H(u)} \\
        & \quad - \frac{S(\tau)}{1-S(\tau)} \frac{\intT \intZ \Phi_k(u,z) dF(u,z)}{1-H(u)}\Bigg)^2 (1-G(u-)) dF(u,z),
    \end{align*}
    with $\Bar{\Phi}_k(t) = \int_t^{\tau} \intZ \Phi_k(u,z) dF(u,z)$.
    Moreover, the estimator
    \begin{align*}
        \hat{\sigma}^2(\Phi_k) & = \frac{(1-\Sn[2](\tau))^{-2}}{n_2} \sum_{i \in \I_2} \indicator{T_i \leq \tau} \Bigg( \Phi_k(T_i,Z_i)\whi \\
        & \quad - \frac{\delta_i}{\hat{Y}_{n_2}(T_i)} \frac{1}{n_2} \sum_{j \in \I_2}  \left( \indicator{T_i \leq T_j } - \frac{\Sn[2](\tau)}{1-\Sn[2](\tau)} \right) \Phi_k(T_j,Z_j) \whj \Bigg)^2,
    \end{align*}
    where $\hat{Y}_{n_2}(t) = \frac{1}{n_2} \sum_{i \in \I_2} \indicator{T_i \geq t}$, is consistent for the variance $\sigma^2(\Phi_k)$.
\end{prop}

\begin{proof}
For all $i \in \I_2$, we introduce the martingale residuals
\begin{align*}
    & M_i^F(t) = N_i(t) - \int_0^t \indicator{T_i\geq u} \frac{dF(u)}{1-F(u-)} \text{, with } N_i(t) = \indicator{T_i \leq t, \delta_i = 1}, \\
    & M_i^G(t) = N_i^G(t) - \int_0^t \indicator{T_i\geq u} \frac{dG(u)}{1-G(u-)} \text{, with } N_i^G(t) = \indicator{T_i \leq t, \delta_i = 0}, \\
    & M_+^{X}(t) = \frac{1}{n_2} \sum_{i \in \I_2} M_i^{X}(t) \text{ for } X = F,G.
\end{align*}
First note that the weights $\whi$, $i \in \I_2$ defined in Equation \eqref{eq::weights2} simplify into
$\whi = \delta_i/(1 - \Gn[2](T_i-))$ because $\Phi_k$ has support $[0,\tau]$. Then
\begin{align*}
    & \sqrt{n_2} (\hat{\theta}_k - \theta_k) \\
    & \quad = \sqrt{n_2}\left(\frac{\intT \intZ \Phi_k(u,z) d\Fn[2](u,z)}{1-\Sn[2](\tau)} - \frac{\intT \intZ \Phi_k(u,z) dF(u,z)}{1-S(\tau)} \right) \\
    & \quad = \frac{\sqrt{n_2}}{1-S(\tau)} \intT \intZ \Phi_k(u,z) d(\Fn[2] - F)(u,z) \\
    & \qquad + \sqrt{n_2} \frac{\Sn[2](\tau) - S(\tau)}{(1-S(\tau))(1-\Sn[2](\tau))}  \intT \intZ \Phi_k(u,z) dF(u,z) \\
    & \qquad  + \sqrt{n_2}  \frac{\Sn[2](\tau) - S(\tau)}{(1-S(\tau))(1-\Sn[2](\tau))} \intT \intZ \Phi_k(u,z) d(\Fn[2] - F)(u,z) \\
    & \quad =: A + B + C.
\end{align*}

From~\citet{lopez_reduction_2007}, $A$ is a Kaplan-Meier integral and we have 
\begin{align*}
    A & = \frac{\sqrt{n_2}}{1-S(\tau)} \frac{1}{n_2} \sum_{i \in \I_2} \left( \Phi_k(T_i, Z_i) \frac{\delta_i \indicator{T_i \leq \tau}}{1-G(T_i-)} - \intT \intZ \Phi_k(u,z) dF(u,z) + \gamma_1(\Phi_k;T_i,\delta_i) \right) \\
    & \quad +  O_{\Prob}(n_2^{-1/2}),
\end{align*}
where for all $i \in \I_2$,
\[ 
    \gamma_1(\Phi_k;T_i,\delta_i) = \intT \frac{\Bar{\Phi}_k(y) dM_i^G(y)}{1-H(y)}\cdot
\]
Moreover, using the martingale decomposition of the Kaplan-Meier estimator $\Sn[2](\tau) - S(\tau) = - S(\tau) \intT (dM_+^{F}(u))/(1-H(u)) + o_{\Prob}(n_2^{-1/2})$ \citep[see for instance][]{andersen_statistical_1993}, we can rewrite $B$ as
\[
    B = - \sqrt{n_2} \frac{S(\tau)}{(1-S(\tau))^2} \intT \frac{dM_+^{F}(u)}{1-H(u)} \intT \intZ \Phi_k(u,z) dF(u,z) + o_{\Prob}(1).
\]
Finally, the $C$ term tends to $0$ in probability as $n_2$ tends to infinity since $\sqrt{n_2} \intT \intZ \Phi_k(u,z) d(\Fn[2] - F)(u,z) = O_{\Prob}(1)$ (see Proposition 2.3.1 in \cite{lopez_reduction_2007}) and $\Sn[2](\tau) - S(\tau) = o_{\Prob}(1)$ from the consistency of the Kaplan-Meier estimator. Finally, we obtain the following centered and i.i.d  representation of $ \sqrt{n_2} (\hat{\theta}_k - \theta_k) $:
\begin{align*}
    \sqrt{n_2} (\hat{\theta}_k - \theta_k) = & \sqrt{n_2} \frac{1}{n_2} \sum_{i \in \I_2} \frac{1}{1-S(\tau)} \Bigg( \Phi_k(T_i, Z_i) \frac{\delta_i \indicator{T_i \leq \tau}}{1-G(T_i-)} \\
    & - \intT \intZ \Phi_k(u,z) dF(u,z) + \gamma_1(\Phi_k;T_i,\delta_i) \\ 
    & - \frac{S(\tau)}{1-S(\tau)} \intT \intZ \Phi_k(u,z) dF(u,z) \intT \frac{dM_i^{F}(u)}{1-H(u)}  \Bigg) + o_{\Prob}(1).
\end{align*}
The convergence in law follows from the central limit theorem. 

Next, we compute the asymptotic variance, denoted $\sigma^2(\Phi_k)$, in the following way:
\begin{align*}
    & (1-S(\tau))^2 \sigma^2(\Phi_k) \\
    & \quad = \Var\left(\Phi_k(T, Z) \frac{\delta \indicator{T \leq \tau}}{1-G(T-)} + \gamma_1(\Phi_k;T,\delta) \right) \\
    & \qquad + \Var\left( \frac{S(\tau)}{1-S(\tau)} \intT \intZ \Phi_k(u,z) dF(u,z) \intT \frac{dM^{F}(u)}{1-H(u)} \right) \\
    & \qquad + 2 ~ \Cov \Bigg(\Phi_k(T, Z) \frac{\delta \indicator{T \leq \tau}}{1-G(T-)} + \gamma_1(\Phi_k;T,\delta),  \frac{S(\tau)}{1-S(\tau)} \intT \intZ \Phi_k(u,z) dF(u,z) \intT \frac{dM^{F}(u)}{1-H(u)}\Bigg) \\
    & \quad =: D + E + 2 F.
\end{align*}
First,
\[
    D = \intT \intZ \frac{\left( \Phi_k(u,z) - \frac{\Bar{\Phi}_k(u)}{1-F(u)} \right)^2}{1-G(u-)} dF(u,z)
\]
by Proposition 2.3.1 in~\citet{lopez_reduction_2007}.
Second, using Theorem 2.4.4 in~\citet{fleming_counting_2005},
\begin{align*}
    E & = \frac{S(\tau)^2}{(1-S(\tau))^2} \left( \intT \intZ \Phi_k(u,z) dF(u,z) \right)^2 \intT \frac{dF(u)}{(1-F(u)) (1-H(u))}\cdot
\end{align*}
Third,
\begin{align*}
    F & = \frac{S(\tau)}{1-S(\tau)} \intT \intZ \Phi_k(u,z) dF(u,z) \\
    & \quad \times \Esp\left[ \left(\Phi_k(T, Z) \frac{\delta \indicator{T \leq \tau}}{1-G(T-)} + \intT \frac{\Bar{\Phi}_k(y) dM^G(y)}{1-H(y)} \right) \intT \frac{dM^{F}(u)}{1-H(u)} \right].
\end{align*}
We then study the $F$ term. From the definition of $M^{F}$ and the Fubini theorem, we have
\begin{align*}
    & \Esp\left[ \Phi_k(T, Z) \frac{\delta \indicator{T \leq \tau}}{1-G(T-)}  \intT \frac{dM^{F}(u)}{1-H(u)} \right] \\
    & \quad = \intT \intZ \frac{\Phi_k(u,z)}{1-H(u)} dF(u,z) - \intT \intZ \Phi_k(u,z) \intT \frac{\indicator{u \geq y}}{1-H(y)} \frac{dF(y)}{1-F(y)} dF(u,z) \\
    & \quad = \intT \intZ \frac{\Phi_k(u,z)}{1-H(u)} dF(u,z) - \intT \frac{\Bar{\Phi}_k(y)}{1-H(y)} \frac{dF(y)}{1-F(y)}\cdot
\end{align*}
In addition, using the fact that $N^G$ and $N^F$ cannot jump at the same time, we have
\begin{align*}
    & \Esp\left[ \intT \frac{\Bar{\Phi}_k(y) dM^G(y)}{1-H(y)}  \intT \frac{dM^{F}(u)}{1-H(u)} \right] \\
    & \quad = \Esp\bigg[ \intT \intT \frac{\Bar{\Phi}_k(y)}{1-H(y)} \frac{1}{1-H(u)} \bigg(- dN^G(y) \indicator{T \geq u} \frac{dF(u)}{1-F(u-)} \\
    & \qquad - dN^F(u) \indicator{T \geq y} \frac{dG(y)}{1-G(y-)} + \indicator{T \geq u \vee y} \frac{dF(u)}{1-F(u-)} \frac{dG(y)}{1-G(y-)} \bigg)\bigg] \\
    & \quad = - \Esp\left[ \intT \frac{\Bar{\Phi}_k(C)}{1-H(C)} \frac{(1-\delta) \indicator{C\leq\tau}}{1-H(u)} \indicator{C \geq u} \frac{dF(u)}{1-F(u-)} \right] \\
    & \qquad - \Esp\left[ \intT \frac{\Bar{\Phi}_k(y)}{1-H(y)} \frac{\delta \indicator{T^*\leq\tau}}{1-H(T^*)} \indicator{T^* \geq y} \frac{dG(y)}{1-G(y-)} \right] \\
    & \qquad + \intT \intT \frac{\Bar{\Phi}_k(y)}{1-H(y)} \frac{1-H(u \vee y)}{1-H(u)}\frac{dF(u)}{1-F(u-)} \frac{dG(y)}{1-G(y-)}\cdot
\end{align*}
We now take the conditional expectation with respect to $C$ for the first term, with respect to $T^*$ for the second term and we use the relations $\Esp[1-\delta\mid C]=1-F(C)$, $\Esp[\delta\mid T^*]=1-G(T^*-)$ in order to obtain the following result
\begin{align*}
\Esp\left[ \intT \frac{\Bar{\Phi}_k(y) dM^G(y)}{1-H(y)}  \intT \frac{dM^{F}(u)}{1-H(u)} \right] & = - \intT \intT \frac{\Bar{\Phi}_k(y)}{1-H(u)} \indicator{y \geq u} \frac{dF(u)}{1-F(u-)} \frac{dG(y)}{1-G(y-)} \\
    & \quad - \intT \intT \frac{\Bar{\Phi}_k(y)}{1-H(y)} \indicator{u \geq y} \frac{dF(u)}{1-F(u-)} \frac{dG(y)}{1-G(y-)} \\
    & \quad + \intT \intT \frac{\Bar{\Phi}_k(y)}{1-H(y)} \frac{1-H(u \vee y)}{1-H(u)}\frac{dF(u)}{1-F(u-)} \frac{dG(y)}{1-G(y-)} \\
    & = 0.
\end{align*}
This proves that $\Esp[ \gamma_1(\Phi_k;T,\delta)\intT (dM^{F}(u))/(1-H(u))]=0$. To conclude, we have shown that
\begin{align*}
    F & = \frac{S(\tau)}{1-S(\tau)} \intT \intZ \Phi_k(u,z) dF(u,z) \left( \intT \intZ \frac{\Phi_k(u,z)}{1-H(u)} dF(u,z) - \intT \frac{\Bar{\Phi}_k(y)}{1-H(y)} \frac{dF(y)}{1-F(y)} \right).
\end{align*}
Finally,
\begin{align*}
    & (1-S(\tau))^2 \sigma^2(\Phi_k) \\
    & \quad = \intT \intZ \frac{\left( \Phi_k(u,z) - \frac{\Bar{\Phi}_k(u)}{1-F(u)} \right)^2}{1-G(u-)} dF(u,z) \\
    & \qquad + \frac{S(\tau)^2}{(1-S(\tau))^2} \left( \intT \intZ \Phi_k(u,z) dF(u,z) \right)^2 \intT \frac{dF(u)}{(1-F(u)) (1-H(u))} \\
    & \qquad + \frac{2~S(\tau)}{1-S(\tau)} \intT \intZ \Phi_k(u,z) dF(u,z) \left( \intT \intZ \frac{\Phi_k(u,z)}{1-H(u)} dF(u,z) - \intT \frac{\Bar{\Phi}_k(y)}{1-H(y)} \frac{dF(y)}{1-F(y)} \right) \\
    & \quad =  \intT \intZ  \Bigg( \frac{\Phi_k(u,z)}{1-G(u-)} - \frac{\Bar{\Phi}_k(u)}{1-H(u)}  - \frac{S(\tau)}{1-S(\tau)} \frac{\intT \intZ \Phi_k(u,z) dF(u,z)}{1-H(u)}\Bigg)^2 (1-G(u-)) dF(u,z).
\end{align*}
The consistency of the variance estimator follows from the consistency of the Kaplan-Meier integrals~\citep[see][]{lopez_reduction_2007}, the consistency of $\hat S$ and the consistency of $\hat Y_{n_2}$.
\end{proof}


\printbibliography

\end{document}


\maketitle

This document brings supplementary simulation results as complements to Section 6 of the main document, for the case where the model of the censoring survival function is misspecified.
The same learning algorithms as in the main document are implemented and evaluated (integrated Kaplan-Meier, integrated Cox, integrated RSF, pseudo-observations and linear model).
We consider the two new simulations schemes below.
\begin{description}
    \item[Scheme A3:] The event times are simulated in a similar way as the scheme \textbf{A} in the main document, with a supplementary uninformative covariate. The event times are simulated according to the following linear model:
    $$
        T_i^* = \tilde{\beta}_0^{\top} Z_i  + \varepsilon_i,
    $$
    where $\tilde{\beta}_0 = (5.5,2.5,2.5,0)^{\top}$, the covariates are denoted $Z_i = (1,Z_i^1,Z_i^2,Z_i^3)^{\top}$ with $Z_i^1,Z_i^2 \sim \mathcal{B}(0.5)$, $Z_i^3 \sim U[-1,1]$ and $\varepsilon_i \sim U[-3,3]$ is a random noise. From this model we obtain the following closed form for the RMST:
    \begin{equation} \label{eq::closeRMST}
        \muT^*(Z) = \Esp[T^* \wedge \tau \mid Z] = \beta_{00} + \beta_{01} Z^1 (1-Z^2) + \beta_{10} Z^2 (1-Z^1) + \beta_{11} Z^1 Z^2,
    \end{equation}
    where $\beta_0 = (\beta_{00},\beta_{01},\beta_{10},\beta_{11})^{\top}$ is computed from $10$ million Monte-Carlo samples. The value of $\tau$ is fixed to $8$ which yields $\beta_0 = (5.48,1.77,1.77,2.5)^{\top}$. 
    The censoring is simulated from a Cox model $\lambda^C(t\mid Z)=\lambda_0^C(t)\exp(3 (Z^3)^2+2 Z^1 Z^2 + (1-Z^1)Z^3)$ 
    with Weibull baseline hazard $\mathcal{W}(\nu,\kappa)$ defined as
    \begin{align*}
        \lambda_0^C(t)=\frac{\nu}{\kappa}\left(\frac{t}{\kappa}\right)^{\nu-1}.
    \end{align*} 
    Note that the survival function of the censoring can be expressed as
    \begin{equation} \label{eq::trueG_A3}
        1-G(t\mid Z) = \exp\left[-\left(\frac{t}{\kappa}\right)^\nu\exp\left(3 (Z^3)^2+2 Z^1 Z^2 + (1-Z^1)Z^3\right)\right]. 
    \end{equation}
    We set $\kappa = 12$, $\nu = 6$, leading to $41 \%$ of censored data. 
    \item[Scheme D:] The event times are simulated according to a Cox model $\lambda(t\mid Z)=\lambda_0(t)\exp(\beta^\top Z)$ with Weibull baseline hazard $\mathcal{W}(2,2)$ and five covariates $Z = (Z^1,Z^2,Z^3,Z^4,Z^5)^{\top}$, where $Z^1,Z^5 \sim \mathcal B(0.4)$, $Z^2 \sim \mathcal N(1,2)$ and $Z^3,Z^4 \sim U[-2,2]$. The Cox regression parameters are set to $\beta = (0.7,-0.4,0.5,0,0)$, so that the variables $4$ and $5$ are uninformative. 
    Note that the survival function can be expressed as
    \[
        S(t\mid Z) = \exp\left[-\left(\frac{t}{\kappa}\right)^\nu\exp(\beta^{\top} Z)\right]. 
    \]
    The censoring is simulated according to another Cox model $\lambda^C(t\mid Z)=\lambda_0^C(t)\exp(-0.4 (Z^3)^2-0.6 Z^1 Z^2 + 0.4(1-Z^1)Z^3)$ 
    with Weibull baseline hazard $\mathcal{W}(2.5,1.5)$, leading to $50 \%$ of censored data.
    Note that the survival function of the censoring can be expressed as
    \begin{equation} \label{eq::trueG_D}
        1-G(t\mid Z) = \exp\left[-\left(\frac{t}{\kappa}\right)^\nu\exp\left(-0.4 (Z^3)^2-0.6 Z^1 Z^2 + 0.4(1-Z^1)Z^3\right)\right]. 
    \end{equation}
    The time horizon $\tau = 2.2$ corresponds to the $86$th percentile of the observed times $T$.
\end{description}


\section{Illustration of the WRSS estimator}

In Section 6.1 of the main document, we illustrated the consistency of our WRSS estimator of $\text{MSE}(\muTtilde)$ when based on well specified models of the censoring survival function.
We recall that $\text{MSE}(\muTtilde) = \Esp\left[(T^* \wedge \tau-\muTtilde(Z))^2 \right]$ and that $\muTtilde$ is the limit of the predictor, as defined in Equation~(3) of the main document. We also recall that $\text{MSE}(\muTtilde)$ can be decomposed into an inseparability and imprecision terms, as shown in Equation~(9) of the main document. 
In order to provide supplementary results for the case where the estimator of the censoring survival function is misspecified, we consider the simulation scheme \textbf{A3} for which the censoring is simulated according to a Cox model with complex relationships between the covariates. 

Two learning models are considered for the prediction of the restricted time to event. The first one is a linear model that is directly fitted on the minimum between the true event times and the time horizon $\tau$, using the correct link function  (see Equation~\eqref{eq::closeRMST}). It is considered as the oracle prediction model. The second one is based on pseudo-observations with linear link function. The latter is implemented without interaction terms, i.e. only the covariates $Z^1$ ,$Z^2$, $Z^3$ are included. 
For the oracle model, the imprecision term of the MSE should vanish as the prediction model is correctly specified and the WRSS should converge to the inseparability term if the estimation of the censoring survival function is correctly specified. On the other hand, for the model based on pseudo-observations, the imprecision term will not vanish and the WRSS should then be larger than with the oracle model. Since the RMST has an explicit form, the inseparability term can be easily computed using Monte Carlo simulations. For the imprecision term of the model based on pseudo-observations, $\muTtilde$ was approximated by a predictor $\muThat$ trained on a sample of size $20,000$ and the expectation was calculated using a million Monte-Carlo simulations. 

As for the estimation of the censoring survival function $1-G$, three models are considered: a Kaplan-Meier method, a Cox model and an RSF model. All three models are misspecified. Indeed, the Kaplan-Meier model does not take covariates into account, while the Cox and RSF models are fitted by including all covariates without interaction terms. 
Thus, we have no guarantee that the WRSS should converge to $\text{MSE}(\muTtilde)$ with either of these estimators.
As a comparison, we also display the WRSS estimator computed on the (oracle) true function $1-G$ given in Equation~\eqref{eq::trueG_A3}, for which Theorem~1 of the main document guarantees the convergence towards $\text{MSE}(\muTtilde)$.

In Figure~\ref{fig::WRSS_dep_A3}, we represent the WRSS based on train and test samples of equal size $100$, $500$ and $1,000$ for the two learning algorithms. The boxplots are obtained from $1,000$ repetitions. 
We clearly observe that the Kaplan-Meier method for the censoring distribution provides biased estimates of the MSE. As for the Cox and RSF models, they show results very similar to the oracle censoring weights, but tend to slightly underestimate the MSE.

\begin{figure}[!ht]
    \centering
    \includegraphics[scale = 0.85]{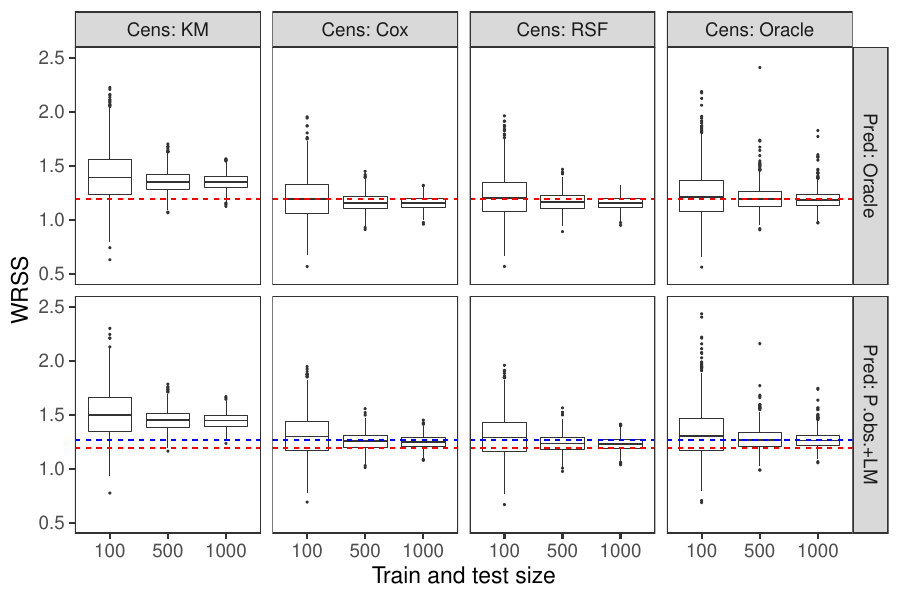}
    \caption{Distribution of $1,000$ replications of the WRSS estimator in the scenario \textbf{A3} and illustration of its convergence towards $\text{MSE}(\muTtilde)$ (see Equation~(9) in the main document), where $\muTtilde$ represents the limit defined in Equation~(3) of the main document. Two learning models are compared. On the top panel, the oracle model~\eqref{eq::closeRMST} is a linear model fitted on the minimum between the true event times and $\tau$, using the correct link function. On the bottom panel, a linear model is implemented based on pseudo-observations, including all covariates without interaction terms. In addition, three censoring estimators are compared. From left to right, a Kaplan-Meier method, a Cox model and an RSF model. The last two are fitted by including all covariates without interaction terms. They are compared to the oracle censoring weights computed with the true function $1-G$ (see Equation~\eqref{eq::trueG_A3}). The red dotted line illustrates the inseparability term. It also represents the $\text{MSE}(\muTtilde)$ for the oracle prediction model, whose imprecision term is null. The blue dotted line represents the $\text{MSE}(\muTtilde)$ for the prediction model based on pseudo-observations, whose imprecision term is non-zero.}
    \label{fig::WRSS_dep_A3}
\end{figure}


\section{Illustration of the IPCW Split Conformal algorithm}

\begin{figure}[!ht]
    \centering
    \includegraphics[scale=0.75]{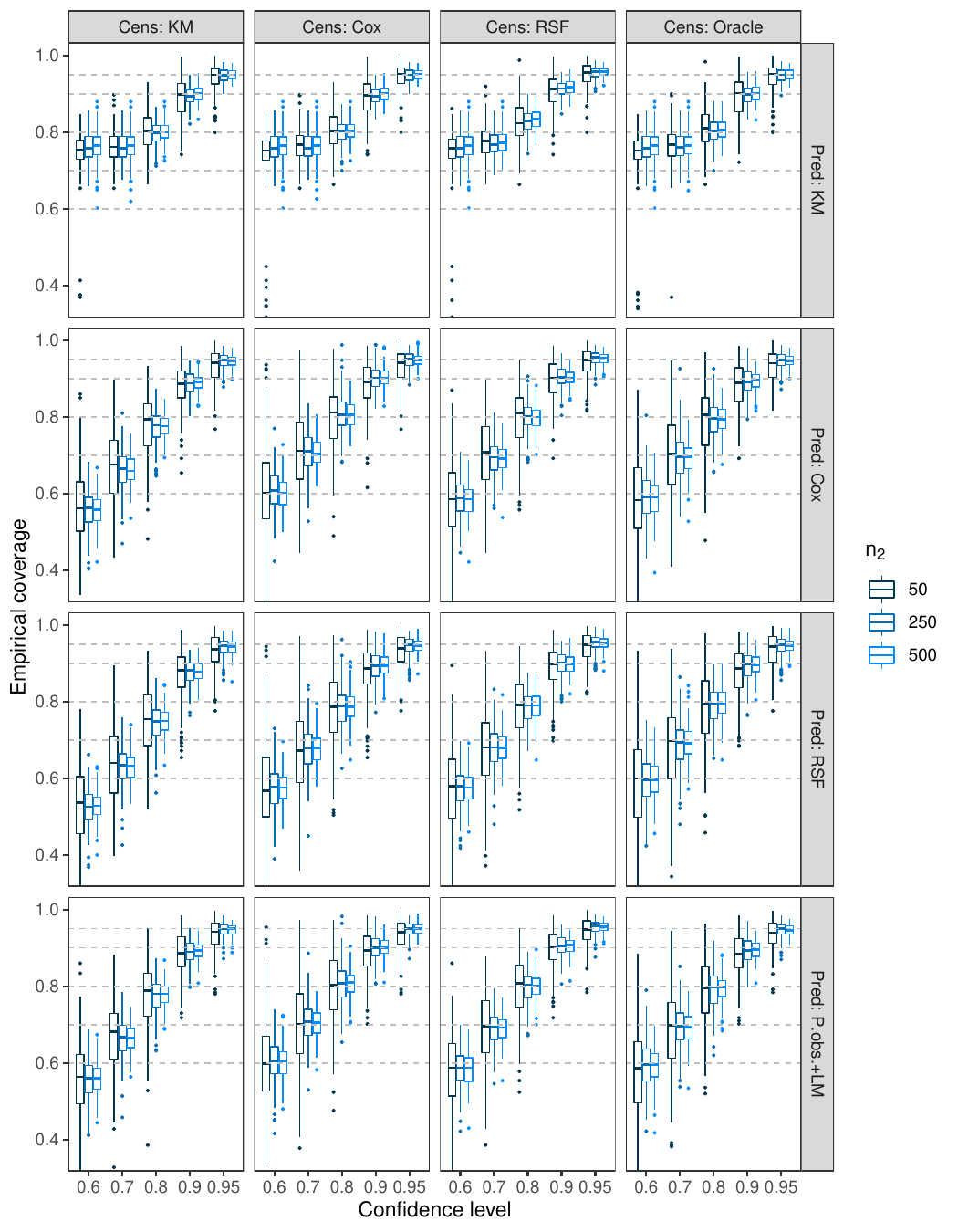}
    \caption{Empirical coverage for the prediction intervals constructed with Algorithm~1 of the main document for four learning models (the Kaplan-Meier estimator, the Cox model, the RSF model and the linear model based on pseudo-observations). Three censoring estimators are considered for the censoring weights (a Kaplan-Meier method, a Cox model and an RSF model) and compared to the oracle weights computed with the true function $1-G$ (see Equation~\eqref{eq::trueG_D}).  All data were simulated according to the scenario \textbf{D}.}
    \label{fig::coverage_fig_D}
\end{figure}

In Section 6.2 of the main document, we illustrated the validity of the coverage of our prediction intervals constructed with the IPCW split conformal algorithm (see Algorithm~1 of the main document). We displayed results for the case where the estimator of the censoring survival function is consistent. 
In order to provide supplementary results for the case where the estimator of the censoring survival function is misspecified, we consider the simulation scheme \textbf{D} for which the censoring is simulated according to a Cox model with complex relationships between the covariates. 

All four learning models introduced at the beginning of Section~6 of the main document are considered for the prediction of the restricted time to event (integrated Kaplan-Meier, integrated Cox, integrated RSF, pseudo-observations and linear model).
As for the estimation of the censoring survival function $1-G$, three models are considered: a Kaplan-Meier method, a Cox model and an RSF model. All three models are misspecified. Indeed, the Kaplan-Meier model does not take covariates into account, while the Cox and RSF models are fitted by including all covariates without interaction terms. 
Thus, we have no guarantee that the empirical coverage should converge to the required level with either of these estimators.
As a comparison, we also display the coverage for the (oracle) true function $1-G$ given in Equation~\eqref{eq::trueG_D}, for which Theorem~2 of the main document guarantees the convergence to the required level.

The coverage of the intervals is assessed in Figure 
\ref{fig::coverage_fig_D}, with $1-\alpha$ equal to $0.6$, $0.7$, $0.8$, $0.9$ or $0.95$. The learning algorithms were trained on samples of size $n=300$, $500$ and $750$ where $n_1$ is fixed to $250$ and $n_2$ takes successively the values $50$, $250$ and $500$. The testing set, on which the empirical coverage is assessed, is of size $m=500$. The simulations were repeated $200$ times.


When the time to event is predicted based on the Kaplan-Meier estimator, we observe the same phenomenon as in the experiments of the main document: for the levels $0.6$ and $0.7$, even with the oracle weights, the empirical coverage converges to a level greater than $1-\alpha$. This is due to the discrete distribution of the residuals when predicting the time to event with the Kaplan-Meier estimator (see Remark~2 of the main document).
At other levels, the empirical coverage converges to the requested value $1-\alpha$, except with the RSF estimator of the censoring survival function, where it converges to a level slightly greater than $1-\alpha$.

For the other learning models, when the censoring survival function is estimated with the Kaplan-Meier estimator, the empirical coverage is lower than the requested level, in particular for low values of $1-\alpha$. For the Cox and RSF estimators of the censoring survival function, the empirical coverage converges to the requested level $1-\alpha$ in most cases. In the case where the time to event is predicted based on the RSF, the empirical coverage converges to a level slightly lower than $1-\alpha$ for low values of $1-\alpha$.

In general, we do not achieve the expected coverage level when using the Kaplan-Meier estimator for the censoring survival function, particularly for low values of $1-\alpha$. However, when using the Cox and RSF models for the censoring distribution, despite these models being misspecified, the empirical coverage remains close to the desired level, though not perfectly optimal.



\section{Illustration of the LOCO variable importance measures}

In Section 6.3 of the main document, we provided illustrations of the performance of the LOCO variable importance measures. 
In particular, the test for global variable importance is shown to be valid under the assumption that the censoring is independent from the time to event and from the covariates (see Theorem~3 of the main document). Thus, we displayed results for cases where the censoring is independent from the covariates, in order for the Kaplan-Meier estimator of the censoring survival function to be consistent. 
In order to provide supplementary results for the case where the censoring is dependent from the covariates, we consider the simulation scheme \textbf{D} for which the censoring is simulated according to a Cox model with complex relationships between the covariates. In the following, for all considered prediction models, the censoring distribution was estimated using the Kaplan-Meier estimator, which is therefore misspecified.
In the scenario \textbf{D}, five variables are considered. Only the first three are used to generate event times according to a Cox model, while the fourth and fifth variables have no impact on the outcome. For all learning algorithms (except the Kaplan-Meier model which does not take covariates into account), we want to test $H_0 : p_k \leq 1/2$ versus $H_1 : p_k > 1/2$, for the variables $k=1, \dots, 5$. As a reminder, $p_k$ is the probability that variable $k$ improves the prediction quality.
However, for each learning algorithm, the values of the $p_k$, $k=1, \dots, 5$, are unknown. Their values are thus approximated via Monte-Carlo simulations. We first simulate a training set $\Dn[1]$ of size $n_1=500$ which remains unchanged throughout the whole simulations (note that Theorem~3 of the main document holds for a fixed $\Dn[1]$). Next, we train the learning algorithms on this data set, simulate $10^5$ pairs $(T^*,Z)$ and compute $p_k$ from the distribution of the corresponding $\Delta_k(Z,T^*)$. Table~\ref{tab::calibrationTable} shows the resulting values, indicating that, for each model, $H_0$ is true for variables 4 and 5 while $H_1$ is true for variables 1,2 and 3.

\begin{table}[!ht]
    \centering
    \begin{tabular}{|l|rrrrr|}
            \hline
            Learning model  &  $p_1$ & $p_2$ & $p_3$ & $p_4$ & $p_5$ \\
            \hline
            Cox  &  0.518 & 0.574 & 0.552 & 0.477 & 0.472 \\
            Random Survival Forest  &  0.539 & 0.587 & 0.585 & 0.472 & 0.491 \\
            Pseudo-observations and linear model  &  0.524 & 0.601 & 0.571 & 0.497 & 0.462 \\
            \hline
        \end{tabular}
    \caption{Values of $p_k$, $k = 1,\dots,5$, for a fixed sample $\Dn[1]$, generated with $n_1=500$ according to the scenario \textbf{D}, for three learning models: the Cox model, the RSF model and the linear model based on pseudo-observations.}
    \label{tab::calibrationTable}
\end{table}

Using the same fixed data set $\Dn[1]$, we empirically assess the calibration and power of our test for global importance by simulating $500$ samples $\Dn[2]$ of size $n_2=500$ and by computing for each one the p-value for the statistical test.
The histograms of those p-values for each value of $k$ and all three prediction algorithms are displayed in Figure~\ref{fig::calibrationPlotD}. 
When $k=4,5$, we observe a skewed distribution of the p-values towards $1$ and $5\%$ rejection rates below $5\%$. This was expected since the $H_0$ hypothesis is composite and, according to Table~\ref{tab::calibrationTable}, $H_0$ is true for $k=4,5$ for all models. When $k=2,3$ ($H_1$ is true) we observe that all three algorithms have a very strong power. However, for $k=1$, where $H_1$ is true for all the models, the p-values are widely spread between 0 and 1 with $5\%$ rejection rate around $6\%$ for all three algorithms. It seems that the test displays very low power when the value of $p_1$ is close to $0.5$.

\begin{figure}[!ht]
    \centering

    \subfloat[$5\%$ rejection rates]{%
    \begin{tabular}{|l|rrr|}
            \hline
            Variable  & Cox & RSF & P.obs.+LM \\
            \hline
            1  & 0.064 & 0.064 & 0.060 \\
            2  & 0.764 & 0.864 & 0.960 \\
            3  & 0.522 & 0.906 & 0.666 \\
            4  & 0.004 & 0.014 & 0.032 \\
            5  & 0.004 & 0.022 & 0.000 \\
            \hline
    \end{tabular}
    \label{fig::reject_rate}
    }

    \subfloat[Distribution of the p-values]{
        \includegraphics[scale=0.55]{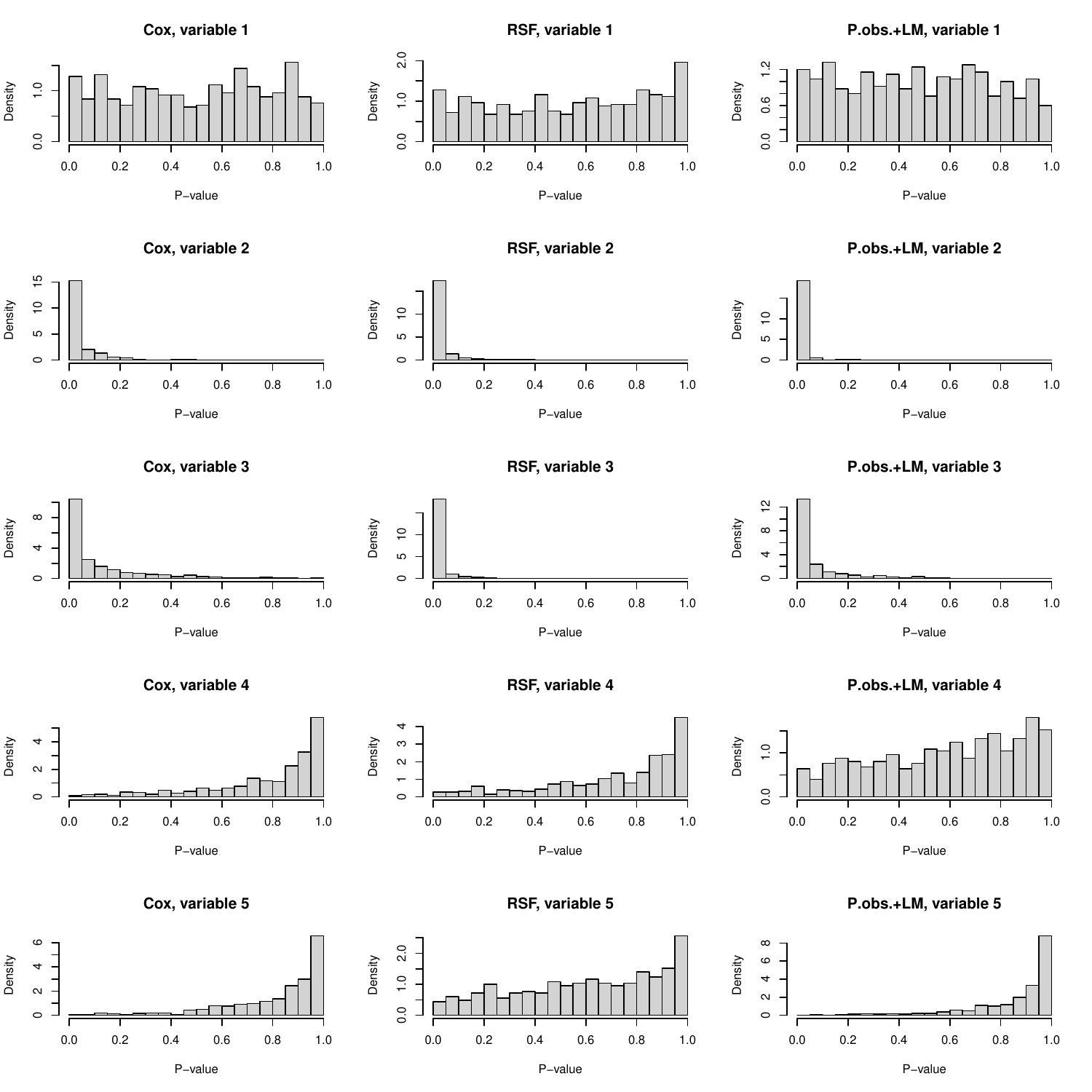} 
        \label{fig::calibration_plot}
    }

    \caption{Distribution of the p-values and $5\%$ rejection rates from $500$ repetitions of the LOCO global variable importance test, for three learning models: the Cox model, the RSF model and the linear model based on pseudo-observations. The sample $\Dn[1]$ was generated with $n_1=500$ and remained fixed while $\Dn[2]$ was simulated $500$ times with $n_2=500$ in order to obtain the distribution of the p-values. All data were simulated according to the scenario \textbf{D}. For all learning models, $H_0$ is true for variables 4 and 5 while $H_1$ is true for variables 1, 2 and 3, see Table~\ref{tab::calibrationTable}.}
    \label{fig::calibrationPlotD}
\end{figure}


\printbibliography